\documentclass[12pt]{amsart}
\usepackage{}
\usepackage{amscd,amsmath,amsthm,amssymb}
\usepackage{pstcol,pst-plot,pst-3d}
\usepackage{color}
\usepackage{pstricks}
\usepackage{lineno,xcolor}
\usepackage{tikz}
\DeclareSymbolFont{largesymbol}{OMX}{yhex}{m}{n}
\DeclareMathAccent{\Widehat}{\mathord}{largesymbol}{"62}

\newpsstyle{fatline}{linewidth=1.5pt}
\newpsstyle{fyp}{fillstyle=solid,fillcolor=verylight}
\definecolor{verylight}{gray}{0.97}
\definecolor{light}{gray}{0.9}
\definecolor{medium}{gray}{0.85}
\definecolor{dark}{gray}{0.6}

 \def\G{{\mathcal G}}

 \def\opn#1#2{\def#1{\operatorname{#2}}} 
 \opn\chara{char} \opn\length{\ell} \opn\pd{pd} \opn\rk{rk}
 \opn\projdim{proj\,dim} \opn\injdim{inj\,dim} \opn\rank{rank}
 \opn\depth{depth} \opn\grade{grade} \opn\height{height}
 \opn\embdim{emb\,dim} \opn\codim{codim}
 
 \opn\Tr{Tr} \opn\bigrank{big\,rank}
 \opn\superheight{superheight}\opn\lcm{lcm}
 \opn\trdeg{tr\,deg}
 \opn\reg{reg} \opn\lreg{lreg} \opn\ini{in} \opn\lpd{lpd}
 \opn\size{size} \opn\sdepth{sdepth}
 \opn\link{link}\opn\fdepth{fdepth}\opn\lex{lex}
 \opn\tr{tr}
 \opn\type{type}
 \opn\Borel{Borel}
\opn\cdeg{cdeg}

 \opn\div{div} \opn\Div{Div} \opn\cl{cl} \opn\Cl{Cl}
 \opn\Spec{Spec} \opn\Supp{Supp} \opn\supp{supp} \opn\Sing{Sing}
 \opn\Ass{Ass} \opn\Min{Min}\opn\Mon{Mon}
 \opn\Ann{Ann} \opn\Rad{Rad} \opn\Soc{Soc}
 \opn\Im{Im} \opn\Ker{Ker} \opn\Coker{Coker} \opn\Am{Am}
 \opn\Hom{Hom} \opn\Tor{Tor} \opn\Ext{Ext} \opn\End{End}
 \opn\Aut{Aut} \opn\id{id}
 
 \opn\nat{nat}
 \opn\pff{pf}
 \opn\Pf{Pf} \opn\GL{GL} \opn\SL{SL} \opn\mod{mod} \opn\ord{ord}
 \opn\Gin{Gin} \opn\Hilb{Hilb}\opn\sort{sort}
 \opn\PF{PF}\opn\Ap{Ap}
 \opn\aff{aff} \opn
 \con{conv} \opn\relint{relint} \opn\st{st}
 \opn\lk{lk} \opn\cn{cn} \opn\core{core} \opn\vol{vol}  \opn\inp{inp} \opn\nilpot{nilpot}
 \opn\link{link} \opn\star{star}\opn\lex{lex}\opn\set{set}
 \opn\width{wd}
 \opn\Fr{F}
 \opn\QF{QF}
 \opn\G{G}
 \opn\type{type}\opn\res{res}
 \opn\gr{gr}
 
  \def\cdeg{deg}

 \def\pot#1#2{#1[\kern-0.28ex[#2]\kern-0.28ex]}

 \opn\dirlim{\underrightarrow{\lim}}
 \opn\inivlim{\underleftarrow{\lim}}
 \def\Implies{\ifmmode\Longrightarrow \else
         \unskip${}\Longrightarrow{}$\ignorespaces\fi}
 \def\implies{\ifmmode\Rightarrow \else
         \unskip${}\Rightarrow{}$\ignorespaces\fi}
 \def\iff{\ifmmode\Longleftrightarrow \else
         \unskip${}\Longleftrightarrow{}$\ignorespaces\fi}

 \let\:=\colon
 \newtheorem{Theorem}{Theorem}[section]
 \newtheorem{Lemma}[Theorem]{Lemma}
 \newtheorem{Corollary}[Theorem]{Corollary}

 \newtheorem{Example}[Theorem]{Example}
 
 \newtheorem{Definition}[Theorem]{Definition}

 \let\epsilon\varepsilon
 \let\kappa=\varkappa
 \def\qed{\ifhmode\textqed\fi
       \ifmmode\ifinner\quad\qedsymbol\else\dispqed\fi\fi}
 \def\textqed{\unskip\nobreak\penalty50
        \hskip2em\hbox{}\nobreak\hfil\qedsymbol
        \parfillskip=0pt \finalhyphendemerits=0}
 \def\dispqed{\rlap{\qquad\qedsymbol}}
 
 \opn\dis{dis}
 \def\pnt{{\raise0.5mm\hbox{\large\bf.}}}
 
 \opn\Lex{Lex}

\begin{document}
\title {Projective dimension and regularity of  edge ideals of  some vertex-weighted  oriented  $m$-partite graphs}

\author {Guangjun Zhu$^{^*}$\!\!\!,  Hong Wang,  Li Xu and Jiaqi Zhang }

\address{Authors¡¯ address:  School of Mathematical Sciences, Soochow
University, Suzhou 215006, P.R. China}
\email{zhuguangjun@suda.edu.cn(Corresponding author:Guangjun Zhu),
\linebreak[4]651634806@qq.com(Hong Wang),1240470845@qq.com(Li Xu),\linebreak[2] zjq7758258@vip.qq.com(Jiaqi \\
Zhang).}

\dedicatory{ }

\begin{abstract}
In this paper we  provide some exact formulas for the projective dimension and the regularity of edge ideals associated to
 three special types of  vertex-weighted  oriented  $m$-partite graphs. These formulas are functions of the weight and number of vertices.
  We also give some examples to show that these formulas are related to direction selection
and the weight of vertices.
\end{abstract}

\thanks{* Corresponding author}



\maketitle

\setcounter{tocdepth}{1}

\section{Introduction}
\hspace{3mm}Let $S=k[x_{1},\dots,x_{n}]$ be a polynomial ring in $n$ variables over a field $k$ and let $I\subset S$ be a homogeneous ideal.
There are two central invariants associated to $I$, the regularity $\mbox{reg}\,(I):=\mbox{max}\{j-i\ |\ \beta_{i,j}(I)\neq 0\}$
 and the projective dimension $\mbox{pd}\,(I):=\mbox{max}\{i\ |\ \beta_{i,j}(I)\neq 0\ \text{for some}\ j\}$, that in a sense, they measure the complexity of computing the graded Betti numbers $\beta_{i,j}(I)$ of $I$. In particular, if $I$ is a
monomial ideal,  its polarization  $I^{\mathcal{P}}$ has the same projective dimension and regularity as $I$ and is squarefree. Thus one can associate $I^{\mathcal{P}}$ to a graph or a hypergraph or  a simplicial complex.
  Many authors have studied the regularity
and Betti numbers of edge ideals of graphs, e.g. \cite{AF1,AF2,BHO,DHS,EK,HT3,J,KM,Z1,Z2,Z3,Z4,Z5,Z6}. Other authors
have studied higher degree generalizations using hypergraphs and clutters \cite{DHS,EK,HT2}
or simplicial complexes \cite{EF,F}.

A {\em directed graph} or {\em digraph} $D$ consists of a finite set $V(D)$ of vertices, together
with a collection $E(D)$ of ordered pairs of distinct points called edges or
arrows. A vertex-weighted digraph is a triplet $D=(V(D), E(D),w)$, where  $w$ is a weight function $w: V(D)\rightarrow \mathbb{N}^{+}$, where $N^{+}=\{1,2,\ldots\}$.
Some times for short we denote the vertex set $V(D)$ and the edge set $E(D)$
by $V$ and $E$ respectively.
The weight of $x_i\in V$ is $w(x_i)$, denoted by $w_i$ or $w_{x_i}$.

The edge ideal of a vertex-weighted digraph was first introduced by Gimenez et al \cite{GBSVV}. Let $D=(V,E,w)$ be a  vertex-weighted digraph with the vertex set $V=\{x_{1},\ldots,x_{n}\}$. We consider the polynomial ring $S=k[x_{1},\dots, x_{n}]$ in $n$ variables over a field $k$. The edge ideal of $D$,  denoted by $I(D)$, is the ideal of $S$ given by
$$I(D)=(x_ix_j^{w_j}\mid  x_ix_j\in E).$$

Edge ideals of weighted digraphs arose in the theory of Reed-Muller codes as initial ideals of vanishing ideals
of projective spaces over finite fields \cite{MPV,PS}.
If a vertex $x_i$ of $D$ is a source (i.e., has only arrows leaving $x_i$) we shall always
assume $w_i=1$ because in this case the definition of $I(D)$ does not depend on the
weight of $x_i$.  If  $w_j=1$ for all $j$, then $I(D)$ is the edge ideal of underlying graph $G$ of $D$.
It has been  studied in the literature  \cite{HT3,T}.
Especially the study of algebraic invariants corresponding to their minimal free resolutions has become popular (see
 \cite{AF1,AF2,BHO,DHS,F, HT2,J,KM,Z1,Z2,Z3,Z4,Z5,Z6}).
In \cite{Z3}, the first three authors  derive some exact formulas for the projective
dimension and regularity of edge ideals of vertex-weighted rooted forests and oriented cycles.
In \cite{Z4}, they  derive some exact formulas for the projective
dimension and regularity of powers of  edge ideals of vertex-weighted rooted forests.
In \cite{Z5,Z6}, they provide some exact formulas for the projective
dimension and regularity of  edge ideals of some  oriented  unicyclic graphs and  cyclic graphs with a common vertex or a common edge.
To the best of our knowledge, little is known about the projective
dimension  and the regularity of $I(D)$ for a  vertex-weighted  oriented graph.

In this article, we are interested in algebraic properties corresponding to the projective
dimension and the regularity of the edge ideals for some  special types of
vertex-weighted  oriented $m$-partite graphs. By using the approaches of Betti splitting and polarization, we derive some exact formulas  for the projective dimension  and the regularity  of edge ideals of these  oriented graphs.
The results are as follows:

\begin{Theorem}
Let $m\geq 2$ be an integer. Assume that $D=(V,E,w)$ is a  vertex-weighted  oriented $m$-partite  graph, its  vertex set $V=\mathop\bigsqcup\limits_{i=1}^{m}V_i$ and its edge set $E=\bigcup\limits_{i=1}^{m-1}E(D_{i})$, where $D_i$ is  a complete bipartite graph and it is also an induced subgraph  of $D$ on $V_{i}\sqcup V_{i+1}$ satisfying the starting point of  every edge of $E(D_{i})$ belongs to $V_{i}$ and  its ending point belongs to $V_{i+1}$ for $1\leq i\leq m-1$. If $w(x)\geq 2$ for any  $x\in V\setminus (V_1\sqcup V_m)$. Then
\begin{itemize}
\item[(1)] $\mbox{reg}\,(I(D))=\sum\limits_{x\in V(D)}w(x)-|V(D)|+2$,
\item[(2)] $\mbox{pd}\,(I(D))=|V(D)|-2$.
\end{itemize}
\end{Theorem}

\begin{Theorem} Let $m\geq 2$ be an integer, and suppose that $D=(V,E,w)$ is a  vertex-weighted oriented  $m$-partite  graph, its  vertex set $V=\mathop\bigsqcup\limits_{i=1}^{m}V_i$ with $|V_1|\leq |V_2|$, its edge set $E=\bigcup\limits_{i=1}^{m-1}E(D_{i})$,  where  $D_1$ is a bipartite graph with the vertex set $\{x_{11},\ldots,x_{1,|V_1|}\}\sqcup \{x_{21},\ldots,x_{2,|V_2|}\}$, the edge set $\{x_{21}x_{11},\ldots,x_{2,|V_1|}x_{1,|V_1|}\}$ and
$D_i$ is  a complete bipartite graph and it is also an  induced subgraph  of $D$ on $V_{i}\sqcup V_{i+1}$ satisfying the starting point of  every edge in $E(D_{i})$ belongs to $V_{i}$ and  its ending point belongs to $V_{i+1}$ for $2\leq i\leq m-1$. If $w(x)\geq 2$ for any  $x\in V\setminus (V_1\sqcup V _m)$. Then
\begin{itemize}
\item[(1)] $\mbox{reg}\,(I(D))=\sum\limits_{x\in V(D)}w(x)-|V(D\setminus V_1)|+1$,
\item[(2)] $\mbox{pd}\,(I(D))=\left\{\begin{array}{ll}
|V(D\setminus V_1)|-2,\ \ &\text{if}\ \  |V_1|<|V_2|,\\
|V(D\setminus V_2)|-1,\ \ &
\text{if}\ \   |V_1|=|V_2|.
\end{array}\right.$
\end{itemize}
\end{Theorem}

\begin{Theorem}
Let $m\geq 3$ be an integer. Assume that $D=(V,E,w)$ is a  vertex-weighted oriented $m$-partite graph, its vertex set $V=\mathop\bigsqcup\limits_{i=1}^{m}V_i$ and its edge set $E=\bigcup\limits_{i=1}^{m}E(D_{i})$, where $D_i$ is  a complete bipartite graph and it is also an induced subgraph  of $D$ on $V_{i}\sqcup V_{i+1}$ satisfying the starting point of  every edge of $E(D_{i})$ belongs to $V_{i}$ and  its ending point belongs to $V_{i+1}$ for $1\leq i\leq m$,  where we stipulate $V_{m+1}=V_1$.
If $w(x)\geq 2$ for all $x\in V$. Then
\begin{itemize}
\item[(1)] $\mbox{reg}\,(I(D))=\sum\limits_{x\in V(D)}w(x)-|V(D)|+1$,
\item[(2)] $\mbox{pd}\,(I(D))=|V(D)|-1$.
\end{itemize}
\end{Theorem}

Our paper is organized as follows. In section $2$, we recall some
definitions and basic facts used in the following sections.
From section $3$ to section $5$,  we
provide some exact formulas for the projective
dimension and the regularity of  the edge
ideals of  three  classes of  vertex-weighted  oriented $m$-partite  graphs  such as Figure $1$.  We also give some examples to show that formulas  for these three types of oriented graphs are related to direction selection
and the weight of vertices.

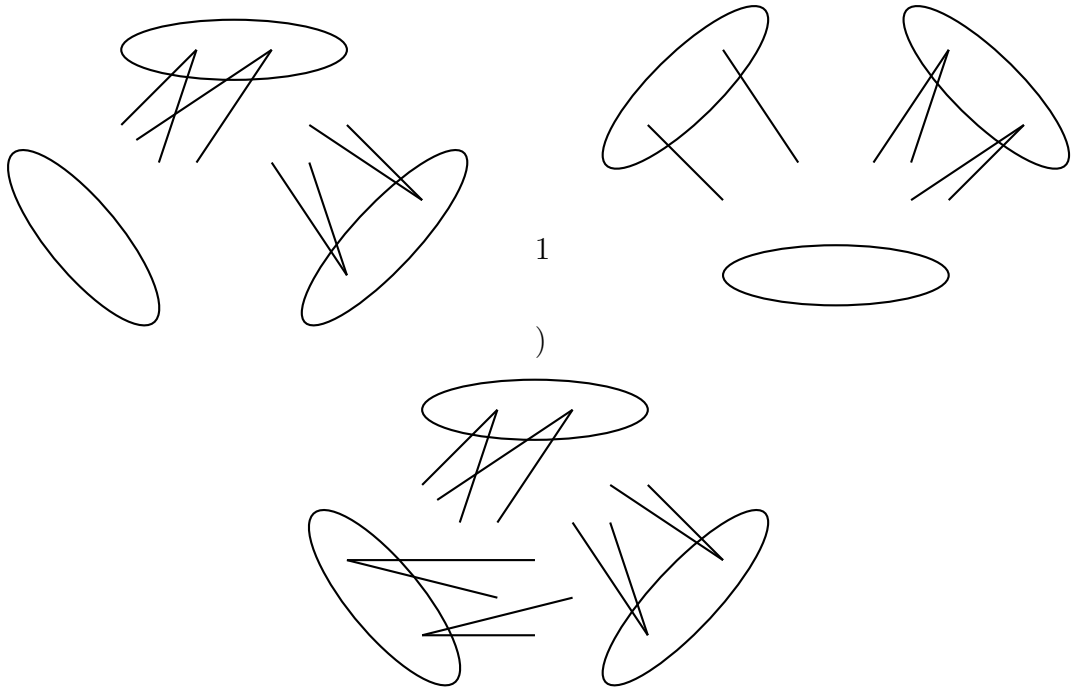
\vspace{0.8cm}
\begin{center}
\begin{tikzpicture}[thick,>=stealth]
\setlength{\unitlength}{1mm}

\thicklines

\pgfpathellipse{\pgfpointxy{-4}{4}}{\pgfpointxy{1.5}{0}}{\pgfpointxy{0}{0.4}}
 \pgfusepath{draw},

\pgfpathellipse{\pgfpointxy{-6}{1.5}}{\pgfpointxy{-1}{1}}{\pgfpointxy{0.1}{0.6}}
 \pgfusepath{draw}

\pgfpathellipse{\pgfpointxy{-2}{1.5}}{\pgfpointxy{1.1}{1}}{\pgfpointxy{0}{0.6}}
 \pgfusepath{draw}

\pgfpathellipse{\pgfpointxy{2.0}{3.5}}{\pgfpointxy{1.1}{0.9}}{\pgfpointxy{0}{0.6}}
 \pgfusepath{draw}

\pgfpathellipse{\pgfpointxy{6.0}{3.5}}{\pgfpointxy{-1.1}{0.9}}{\pgfpointxy{0}{0.6}}
 \pgfusepath{draw}

\pgfpathellipse{\pgfpointxy{4.0}{1}}{\pgfpointxy{1.5}{0}}{\pgfpointxy{0}{0.4}}
 \pgfusepath{draw}

\put(-69,22){$x_{11}$}
\put(-65,20){\circle*{1.5}}

\put(-56,7){$x_{12}$}
\put(-55,10){\circle*{1.5}}

\put(-52,40){$x_{21}$}
\put(-45,40){\circle*{1.5}}

\put(-33,40){$x_{22}$}
\put(-35,40){\circle*{1.5}}

\put(-16,22){$x_{31}$}
\put(-15,20){\circle*{1.5}}

\put(-28,7){$x_{32}$}
\put(-25,10){\circle*{1.5}}

\put(11,27){$x_{11}$}
\put(15,30){\circle*{1.5}}

\put(24,41){$x_{12}$}
\put(25,40){\circle*{1.5}}

\put(28,9){$x_{21}$}
\put(35,10){\circle*{1.5}}

\put(47,9){$x_{22}$}
\put(45,10){\circle*{1.5}}

\put(51,42){$x_{31}$}
\put(55,40){\circle*{1.5}}

\put(64,26){$x_{32}$}
\put(65,30){\circle*{1.5}}

\put(-65,20){\vector(1,1){11}}
\draw[solid](-5.5,3)--(-4.5,4);
\put(-65,20){\vector(3,2){13}}
\draw[solid](-5.3,2.8)--(-3.5,4);

\put(-55,10){\vector(1,3){6}}
\draw[solid](-5,2.5)--(-4.5,4);
\put(-55,10){\vector(2,3){11}}
\draw[solid](-4.5,2.5)--(-3.5,4);

\put(-45,40){\vector(3,-2){20}}
\draw[solid](-3,3)--(-1.5,2);
\put(-45,40){\vector(2,-3){11}}
\draw[solid](-3.5,2.5)--(-2.5,1);

\put(-35,40){\vector(1,-1){11}}
\draw[solid](-2.5,3)--(-1.5,2);
\put(-35,40){\vector(1,-3){5}}
\draw[solid](-3,2.5)--(-2.5,1);

\put(35,10){\vector(-1,1){10}}
\draw[solid](2.5,2)--(1.5,3);
\put(45,10){\vector(-2,3){10}}
\draw[solid](3.5,2.5)--(2.5,4);

\put(35,10){\vector(2,3){10}}
\draw[solid](4.5,2.5)--(5.5,4);
\put(35,10){\vector(3,2){20}}
\draw[solid](5,2)--(6.5,3);

\put(45,10){\vector(1,3){5}}
\draw[solid](5,2.5)--(5.5,4);
\put(45,10){\vector(1,1){12}}
\draw[solid](5.5,2)--(6.5,3);

\put(-40,-0){$(a)$};
\put(40,-0){$(b)$};

\end{tikzpicture}
\end{center}

\vspace{0.5cm}
\begin{center}
\begin{tikzpicture}[thick,>=stealth]
\setlength{\unitlength}{1mm}

\thicklines
\pgfpathellipse{\pgfpointxy{0}{-2}}{\pgfpointxy{1.5}{0}}{\pgfpointxy{0}{0.4}}
 \pgfusepath{draw}

\pgfpathellipse{\pgfpointxy{-2}{-4.5}}{\pgfpointxy{-1}{1}}{\pgfpointxy{0.1}{0.6}}
 \pgfusepath{draw}

\pgfpathellipse{\pgfpointxy{2}{-4.5}}{\pgfpointxy{1.1}{1}}{\pgfpointxy{0}{0.6}}
 \pgfusepath{draw}

\put(-12,-20){$x_{11}$}
\put(-5,-20){\circle*{1.5}}

\put(7,-20){$x_{12}$}
\put(5,-20){\circle*{1.5}}

\put(-29,-37){$x_{21}$}
\put(-25,-40){\circle*{1.5}}

\put(-16,-54){$x_{22}$}
\put(-15,-50){\circle*{1.5}}

\put(10,-54){$x_{31}$}
\put(15,-50){\circle*{1.5}}

\put(24,-38){$x_{32}$}
\put(25,-40){\circle*{1.5}}

\put(-25,-40){\vector(1,1){11}}
\draw[solid](-1.5,-3)--(-0.5,-2);
\put(-25,-40){\vector(3,2){13}}
\draw[solid](-1.3,-3.2)--(0.5,-2);

\put(-15,-50){\vector(1,3){6}}
\draw[solid](-1,-3.5)--(-0.5,-2);
\put(-15,-50){\vector(2,3){11}}
\draw[solid](-0.5,-3.5)--(0.5,-2);

\put(-5,-20){\vector(3,-2){20}}
\draw[solid](1,-3)--(2.5,-4);
\put(-5,-20){\vector(2,-3){11}}
\draw[solid](0.5,-3.5)--(1.5,-5);

\put(5,-20){\vector(1,-1){11}}
\draw[solid](1.5,-3)--(2.5,-4);
\put(5,-20){\vector(1,-3){5}}
\draw[solid](1,-3.5)--(1.5,-5);

\put(25,-40){\vector(-1,0){25}}
\draw[solid](0,-4)--(-2.5,-4);
\put(25,-40){\vector(-4,-1){20}}
\draw[solid](0.5,-4.5)--(-1.5,-5);

\put(15,-50){\vector(-4,1){20}}
\draw[solid](-0.5,-4.5)--(-2.5,-4);
\put(15,-50){\vector(-1,0){15}}
\draw[solid](0,-5)--(-1.5,-5);

\put(-4,-62){$(c)$};
\put(-8,-70){$Figure\hspace{3mm} 1$};
\end{tikzpicture}
\end{center}

\newpage

For all unexplained terminology and additional information, we refer to \cite{JG} (for the theory
of digraphs), \cite{BM} (for graph theory), and \cite{HH} (for the theory of edge ideals of graphs and
monomial ideals). We greatfully acknowledge the use of the computer algebra system CoCoA (\cite{C}) for our experiments.

\medskip

\section{Preliminaries }
\vspace{3mm}In this section, we gather together the needed  definitions and basic facts, which will
be used throughout this paper. However, for more details, we refer the reader to \cite{AF2,BM,FHT,HT1,HH,J,JG,MPV,PRT,Z1,Z3}.

A {\em directed graph} or {\em digraph} $D$ consists of a finite set $V(D)$ of vertices, together
with a collection $E(D)$ of ordered pairs of distinct points called edges or
arrows. If $\{u,v\}\in E(D)$ is an edge, we write $uv$ for $\{u,v\}$, which is denoted to be the directed edge
where the direction is from $u$ to $v$ and $u$ (resp. $v$) is called the {\em starting}  point (resp. the {\em ending} point).
An {\em  oriented} graph is a  directed  graph  having no bidirected edges (i.e. each pair of vertices is joined by a single edge having a unique direction). In other words, an oriented graph $D$ is a simple graph $G$ together with an orientation
of its edges. We call $G$ the underlying graph of $D$.

Every concept that is valid for graphs automatically applies to digraphs too.  A digraph is said to be connected if
its underlying graph is connected. A digraph $H$ is called an induced subgraph of  a digraph $D$  if  $V(H)\subseteq V(D)$,
and for any $x,y\in V(H)$,  $xy$ is an edge of $H$ if and only if  $xy$ is an edge of $D$. For  $P\subset V(D)$, we  denote
$D\setminus P$ the induced subgraph of $D$ obtained by removing the vertices in $P$ and the
edges incident to these vertices. If $P=\{x\}$ consists of a single element, then we write $D\setminus x$ for $D\setminus \{x\}$.
The induced subgraph of $D$ over a subset $W\subset V(G)$ is  a graph with the  vertex set $W$ and the edge set $\{uv\in E(G)\mid u,v \in W\}$.
For $U\subseteq E(D)$, we define $D\setminus U$ to be the subgraph of $D$ with all edges  in $U$  deleted (but its vertices remained).
When $U=\{e\}$ consists of a single edge, we write $D\setminus e$ instead of $D\setminus \{e\}$.
An {\em  oriented}  path or {\em  oriented}  cycle  is an orientation of a
path or cycle in which each vertex dominates its successor in the sequence.
Let $G=(V,E)$ be a finite simple graph on the vertex set $\{x_1,\ldots, x_n\}$, the whisker graph
$G^{*}$ of $G$ is the graph with the vertex set $V\cup \{y_1,\ldots, y_{\ell}\}$ and the
edge set $E(G^{*})=E\cup \{x_{ij}y_i\mid 1\leq j\leq \ell\}$, where $\ell\leq n$ and these $x_{ij}$ are different from each other.
Let $m$ be an integer, a graph $G=(V,E)$ is called $m$-partite if if the set of all its vertices can
be partitioned into $m$ subsets $V_1,\ldots,V_m$, in such a way that
any edge of graph $G$ connects vertices from different subsets. The
terms bipartite graph and tripartite graph are used to describe $m$-partite graphs for $m$  equal to $2$ and $3$, respectively. A $m$-partite
graph is called complete if any vertex $v\in V$ is adjacent to all
vertices not belonging to the same partition as $v$.
 Unless specifically stated, an  oriented bipartite graph with vertex set $V=V_1\sqcup V_2$ in this article
 is a  bipartite graph  in which all edges  are oriented  from the  vertex in $V_1$ to the  vertex in $V_2$.

 A vertex-weighted oriented graph is a triplet $D=(V(D), E(D),w)$, where $V(D)$ is the  vertex set,
$E(D)$ is the edge set and $w$ is a weight function $w: V(D)\rightarrow \mathbb{N}^{+}$, where $N^{+}=\{1,2,\ldots\}$.
Some times for short we denote the vertex set $V(D)$ and edge set $E(D)$
by $V$ and $E$ respectively.
The weight of $x_i\in V$ is $w(x_i)$, denoted by $w_i$  or $w_{x_i}$.
Given  a  vertex-weighted oriented graph $D=(V,E,w)$ with the vertex set $V=\{x_{1},\ldots,x_{n}\}$, we  consider the polynomial ring $S=k[x_{1},\dots, x_{n}]$ in $n$ variables over a field $k$. The edge ideal of $D$,  denoted by $I(D)$, is the ideal of $S$ given by
$$I(D)=(x_ix_j^{w_j}\mid  x_ix_j\in E).$$

If a vertex $x_i$ of $D$ is a source (i.e., has only arrows leaving $x_i$) we shall always
assume $w_i=1$ because in this case the definition of $I(D)$ does not depend on the
weight of $x_i$.

\medskip
For any homogeneous ideal $I$ of the polynomial ring  $S=k[x_{1},\dots,x_{n}]$, there exists a {\em graded
minimal finite free resolution}

\vspace{3mm}
$0\rightarrow \bigoplus\limits_{j}S(-j)^{\beta_{p,j}(M)}\rightarrow \bigoplus\limits_{j}S(-j)^{\beta_{p-1,j}(M)}\rightarrow \cdots\rightarrow \bigoplus\limits_{j}S(-j)^{\beta_{0,j}(M)}\rightarrow I\rightarrow 0,$
where the maps are exact, $p\leq n$, and $S(-j)$ is the $S$-module obtained by shifting
the degrees of $S$ by $j$. The number
$\beta_{i,j}(I)$, the $(i,j)$-th graded Betti number of $I$, is
an invariant of $I$ that equals the number of minimal generators of degree $j$ in the
$i$th syzygy module of $I$.
Of particular interests are the following invariants which measure the ¡°size¡± of the minimal graded
free resolution of $I$.
The projective dimension of $I$, denoted pd\,$(I)$, is defined to be
$$\mbox{pd}\,(I):=\mbox{max}\,\{i\ |\ \beta_{i,j}(I)\neq 0\}.$$
The regularity of $I$, denoted $\mbox{reg}\,(I)$, is defined by
$$\mbox{reg}\,(I):=\mbox{max}\,\{j-i\ |\ \beta_{i,j}(I)\neq 0\}.$$

We now derive some formulas for $\mbox{pd}\,(I)$  and $\mbox{reg}\,(I)$ in some special cases by using some
tools developed in \cite{FHT}.

\begin{Definition} \label{bettispliting}Let $I$  be a monomial ideal, and suppose that there exist  monomial
ideals $J$ and $K$ such that $\mathcal{G}(I)$ is the disjoint union of $\mathcal{G}(J)$ and $\mathcal{G}(K)$, where $\mathcal{G}(I)$ denotes the unique minimal set of monomial generators of $I$. Then $I=J+K$
is a {\em Betti splitting} if
$$\beta_{i,j}(I)=\beta_{i,j}(J)+\beta_{i,j}(K)+\beta_{i-1,j}(J\cap K)\hspace{2mm}\mbox{for all}\hspace{2mm}i,j\geq 0,$$
where $\beta_{i-1,j}(J\cap K)=0\hspace{2mm}  \mbox{if}\hspace{2mm} i=0$.
\end{Definition}

In  \cite{FHT}, the authors describe some sufficient conditions for an
ideal $I$ to have a Betti splitting. We need  the following lemma.

\begin{Lemma}\label{lem1}(\cite[Corollary 2.7]{FHT}).
Suppose that $I=J+K$ where $\mathcal{G}(J)$ contains all
the generators of $I$ divisible by some variable $x_{i}$ and $\mathcal{G}(K)$ is a nonempty set containing
the remaining generators of $I$. If $J$ has a linear resolution, then $I=J+K$ is a Betti
splitting.
\end{Lemma}

When $I$ is a Betti
splitting ideal, Definition \ref{bettispliting} implies the following results:
\begin{Corollary} \label{cor1}
If $I=J+K$ is a Betti splitting ideal, then
\begin{itemize}
 \item[(1)]$\mbox{reg}\,(I)=\mbox{max}\{\mbox{reg}\,(J),\mbox{reg}\,(K),\mbox{reg}\,(J\cap K)-1\}$,
 \item[(2)] $\mbox{pd}\,(I)=\mbox{max}\{\mbox{pd}\,(J),\mbox{pd}\,(K),\mbox{pd}\,(J\cap K)+1\}$.
\end{itemize}
\end{Corollary}

The following lemmas is often used in this article.
\begin{Lemma}
\label{lem2}(\cite[Lemma 1.3]{HTT}) Let $R$ be a polynomial ring over a field and let $I$ be a proper non-zero homogeneous
ideal in $R$. Then
\begin{itemize}
\item[(1)] $\mbox{pd}\,(I)=\mbox{pd}\,(R/I)-1$,
\item[(2)] $\mbox{reg}\, (I)=\mbox{reg}\,(R/I)+1$.
\end{itemize}
\end{Lemma}

\begin{Lemma}
\label{lem3}(\cite[Lemma 2.2 and  Lemma 3.2 ]{HT1})
Let $S_{1}=k[x_{1},\dots,x_{m}]$, $S_{2}=k[x_{m+1},\dots,x_{n}]$ and $S=k[x_{1},\dots,x_{n}]$ be three polynomial rings, $I\subseteq S_{1}$ and
$J\subseteq S_{2}$ be two proper non-zero homogeneous  ideals.  Then
\begin{itemize}
\item[(1)] $\mbox{pd}\,(S/(I+J))=\mbox{pd}\,(S_{1}/I)+\mbox{pd}\,(S_{2}/J)$,
\item[(2)] $\mbox{reg}\,(S/(I+J))=\mbox{reg}\,(S_{1}/I)+\mbox{reg}\,(S_{2}/J)$.
\end{itemize}
\end{Lemma}

\medskip
From Lemma \ref{lem2} and Lemma \ref{lem3}, we have
\begin{Lemma}
\label{lem4}(\cite[Lemma 3.1]{Z1})
Let $S_{1}=k[x_{1},\dots,x_{m}]$ and $S_{2}=k[x_{m+1},\dots,x_{n}]$ be two polynomial rings, $I\subseteq S_{1}$ and
$J\subseteq S_{2}$ be two non-zero homogeneous  ideals. Then
\begin{itemize}
\item[(1)]$\mbox{pd}\,(I+J)=\mbox{pd}\,(I)+\mbox{pd}\,(J)+1$,
\item[(2)]$\mbox{reg}\,(I+J)=\mbox{reg}\,(I)+\mbox{reg}\,(J)-1$.
\end{itemize}
\end{Lemma}

\medskip
  Let $\mathcal{G}(I)$ denote the minimal set of generators of a monomial ideal $I\subset S$
 and let $u\in S$ be a monomial, we set $\mbox{supp}(u)=\{x_i: x_i|u\}$. If $\mathcal{G}(I)=\{u_1,\ldots,u_m\}$, we set $\mbox{supp}(I)=\bigcup\limits_{i=1}^{m}\mbox{supp}(u_i)$. The following lemma is well known.
 \begin{Lemma}\label{lem5}
  Let  $I, J=(u)$ be two monomial ideals  such that $\mbox{supp}\,(u)\cap \mbox{supp}\,(I)=\emptyset$. If the degree of monomial $u$  is  $d$. Then
\begin{itemize}
\item[(1)] $\mbox{reg}\,(J)=d$,
\item[(2)]$\mbox{reg}\,(JI)=\mbox{reg}\,(I)+d$,
\item[(3)]$\mbox{pd}\,(JI)=\mbox{pd}\,(I)$.
\end{itemize}
\end{Lemma}

\medskip
\begin{Definition} \label{polarization}
Suppose that $u=x_1^{a_1}\cdots x_n^{a_n}$ is a monomial in $S$. We define the {\it polarization} of $u$ to be
the squarefree monomial $$\mathcal{P}(u)=x_{11}x_{12}\cdots x_{1a_1} x_{21}\cdots x_{2a_2}\cdots x_{n1}\cdots x_{na_n}$$
in the polynomial ring $S^{\mathcal{P}}=k[x_{ij}\mid 1\leq i\leq n, 1\leq j\leq a_i]$.
If $I\subset S$ is a monomial ideal with $\mathcal{G}(I)=\{u_1,\ldots,u_m\}$,  the  {\it polarization}
of $I$,  denoted by $I^{\mathcal{P}}$, is defined as:
$$I^{\mathcal{P}}=(\mathcal{P}(u_1),\ldots,\mathcal{P}(u_m)),$$
which is a squarefree monomial ideal in the polynomial ring $S^{\mathcal{P}}$.
\end{Definition}

\medskip
Here is an example of how polarization works.

\begin{Example} \label{exm1} Let $I(D)=(x_3x_1^{2},x_4x_2^2,x_3x_5^{2},x_3x_6^{2},x_4x_5^{2},x_4x_6^{2})$ be the edge ideal of a vertex-weighted digraph $D=(V,E,w)$, where  $V=\mathop\bigsqcup\limits_{j=1}^{3}V_j$ with $V_{1}=\{x_1,x_2\}$, $V_{2}=\{x_3,x_4\}$ and $V_{3}=\{x_5,x_6\}$. Then the polarization $I(D)^{\mathcal{P}}$ of $I(D)$ is the ideal  $(x_{31}x_{11}x_{12},
x_{41}x_{21}x_{22},x_{31}x_{51}x_{52},x_{31}x_{61}x_{62},x_{41}x_{51}x_{52},x_{41}x_{61}x_{62})$.
\end{Example}

\medskip
A monomial ideal $I$ and its polarization $I^{\mathcal{P}}$ share many homological and
algebraic properties.  The following is a very useful property of polarization.

\begin{Lemma}
\label{lem6}(\cite[Corollary 1.6.3]{HH}) Let $I\subset S$ be a monomial ideal and $I^{\mathcal{P}}\subset S^{\mathcal{P}}$ its polarization.
Then
\begin{itemize}
\item[(1)] $\beta_{ij}(I)=\beta_{ij}(I^{\mathcal{P}})$ for all $i$ and $j$,
\item[(2)] $\mbox{reg}\,(I)=\mbox{reg}\,(I^{\mathcal{P}})$,
\item[(3)] $\mbox{pd}\,(I)=\mbox{pd}\,(I^{\mathcal{P}})$.
\end{itemize}
\end{Lemma}

The following lemma can be used for computing the projective dimension and the
regularity of an ideal.
\begin{Lemma}
\label{lem7}(\cite[Lemma 1.1 and Lemma 1.2]{HTT})  Let\ \ $0\rightarrow A \rightarrow  B \rightarrow  C \rightarrow 0$\ \  be a short exact sequence of finitely generated graded $S$-modules.
Then
\begin{itemize}
\item[(1)]$\mbox{reg}\,(B)= \mbox{reg}\,(C)$ if $\mbox{reg}\,(A)\leq\mbox{reg}\,(C)$,
\item[(2)]$\mbox{pd}\,(B)= \mbox{pd}\,(A)$ if $\mbox{pd}\,(A)\geq \mbox{pd}\,(C)$.
\end{itemize}
\end{Lemma}

\medskip

\medskip

\section{Projective dimension and regularity of  edge ideals of the first class of  vertex-weighted oriented $m$-partite  graphs}

In this section, we will provide some exact formulas for the projective
dimension and the regularity of  the edge ideals of  a  class of  vertex-weighted oriented $m$-partite  graphs with the vertex set $V=\mathop\bigsqcup\limits_{i=1}^{m}V_i$ and the edge set $E=\bigcup\limits_{i=1}^{m-1}E(D_{i})$, where $D_i$ is  a complete bipartite graph and it is also an induced subgraph  of $D$ on $V_{i}\sqcup V_{i+1}$ satisfying the starting point of  every edge of $E(D_{i})$ belongs to $V_{i}$ and  its ending point belongs to $V_{i+1}$ for $1\leq i\leq m-1$.
 We also give some examples to show that these formulas  are related to direction selection
and the weight of vertices.  We shall start
with the following  lemma.

\begin{Lemma}
\label{lem8} (\cite[Theorem 3.2]{Z6}) Let $D=(V(D),E(D),w)$ be a vertex-weighted oriented complete bipartite graph.
Then
\begin{itemize}
\item[(1)] $\mbox{reg}\,(I(D))=\sum\limits_{x\in V(D)}w(x)-|V(D)|+2$,
\item[(2)] $\mbox{pd}\,(I(D))=|V(D)|-2$.
\end{itemize}
\end{Lemma}

\medskip
Now we are ready to present the main results of this section.
\begin{Theorem}\label{thm1}
Let $m\geq 2$ be an integer, and assume that $D=(V,E,w)$ is a  vertex-weighted  oriented $m$-partite  graph, its  vertex set $V=\mathop\bigsqcup\limits_{i=1}^{m}V_i$ and its edge set $E=\bigcup\limits_{i=1}^{m-1}E(D_{i})$, where $D_i$ is  a complete bipartite graph and it is also an induced subgraph  of $D$ on $V_{i}\sqcup V_{i+1}$ satisfying the starting point of  every edge in $E(D_{i})$ belongs to $V_{i}$ and  its ending point belongs to $V_{i+1}$ for $1\leq i\leq m-1$. If $w(x)\geq 2$ for any  $x\in V\setminus (V_1\sqcup V_m)$. Then
$$\mbox{reg}\,(I(D))=\sum\limits_{x\in V(D)}w(x)-|V(D)|+2.$$
\end{Theorem}
\begin{proof}
Let $V_{i}=\{x_{i1},\ldots,x_{i,\,t_{i}}\}$ for $1\leq i\leq m$, then
\begin{eqnarray*}
I(D)&=&(x_{11}x_{21}^{w_{21}}\!,x_{11}x_{22}^{w_{22}}\!,\ldots,x_{11}x_{2,\,t_2}^{w_{2,t_2}}\!, x_{12}x_{21}^{w_{21}}\!,\ldots,x_{12}x_{2,\,t_2}^{w_{2,t_2}}\!,\ldots, x_{1,\,t_1}x_{21}^{w_{21}}\!,\ldots,\\
& &x_{1,\,t_1}x_{2,\,t_2}^{w_{2,t_2}}\!,x_{21}x_{31}^{w_{31}}\!,\ldots, x_{21}x_{3,\,t_3}^{w_{3,t_3}}\!, x_{22}x_{31}^{w_{31}}\!,\ldots,
x_{22}x_{3,\,t_3}^{w_{3,t_3}}\!,\ldots,x_{2,\,t_2}x_{31}^{w_{31}}\!,\ldots, \\
& &x_{2,\,t_2}x_{3,\,t_3}^{w_{3,t_3}}\!,\ldots, x_{m-1,1}x_{m1}^{w_{m1}}\!,\ldots, x_{m-1,1}x_{m,\,t_m}^{w_{m,t_m}}\!,x_{m-1,2}x_{m1}^{w_{m1}}\!,\ldots, x_{m-1,2}x_{m,\,t_m}^{w_{m,t_m}}\!,\\
& &x_{m-1,\,t_{m-1}}x_{m1}^{w_{m1}}\!,\ldots,x_{m-1,\,t_{m-1}}x_{m,\,t_m}^{w_{m,t_m}}).
\end{eqnarray*}
We apply induction on $m$. The case $m=2$ follows from Lemma \ref{lem8} (1).

 Now we assume that $m\geq3$. Consider the following short exact sequences

\begin{gather*}
\hspace{0.5cm}\begin{matrix}
 0 & \longrightarrow & \frac{S}{(I(D)\,:\,x_{m1}^{w_{m1}})}(-w_{m1}) & \stackrel{ \cdot x_{w_{m1}}} \longrightarrow & \frac{S}{I(D)} &\longrightarrow & \frac{S}{J_1} & \longrightarrow & 0 \\
 \\
 0 & \longrightarrow & \frac{S}{(J_1\,:\,x_{m2}^{w_{m2}})}(-w_{m2}) & \stackrel{ \cdot x_{m2}^{w_{m2}}} \longrightarrow  &   \frac{S}{J_1} & \longrightarrow & \frac{S}{J_2} & \longrightarrow & 0  &\hspace{1.0cm} (\ddagger) \\
 \\
     &  &\vdots&  &\vdots&  &\vdots&  &\\
     \\
 0&  \longrightarrow & \frac{S}{(J_{t_m-1}\,:\,x_{m,\,t_m}^{w_{m,t_m}})}(-w_{m,t_m})& \stackrel{ \cdot x_{m,t_m}^{w_{m,t_m}}} \longrightarrow & \frac{S}{J_{t_m-1}}& \longrightarrow & \frac{S}{J_{t_m}}& \longrightarrow & 0\\
 \end{matrix}\quad
\end{gather*}
where $J_{i}=I(D)+(x_{m1}^{w_{m1}},\ldots,x_{mi}^{w_{mi}})$ for $1\leq i\leq t_m$. We prove this argument in the
following two steps.

\vspace{0.3cm}
(1)  We first prove  $\mbox{reg}\,( J_{t_m})=\sum\limits_{x\in V(D)}w(x)-|V(D)|+2$.

In fact,
$J_{t_m}=I(D)+(x_{m1}^{w_{m1}},\ldots,x_{m,\,t_m}^{w_{m,\,t_m}})=I(D\setminus V_m)+(x_{m1}^{w_{m1}},\ldots,x_{m,\,t_m}^{w_{m,\,t_m}})$,
where $I(D\setminus V_m)$ is the edge ideal of the  induced subgraph $D\setminus V_m$ of $D$  on the set $V\setminus V_m$ obtained
by removing the vertices in $V_m$ and the edges incident to these vertices. Let  $K=(x_{m1}^{w_{m1}},x_{m2}^{w_{m2}},\ldots,x_{m,\,t_m}^{w_{m,t_m}})$,
then  the variables appearing in $I(D\setminus V_m)$ and $K$ are different, then by induction hypothesis on $m$  and Lemma \ref{lem4} (2), we get
\begin{eqnarray*}
\mbox{reg}\,(J_{t_m})&=&\mbox{reg}\,(I(D\setminus V_m))+\mbox{reg}\,(K)-1\\
&=&[\sum\limits_{x\in V(D\setminus V_m)}w(x)-|V(D\setminus V_m)|+2]+[\sum\limits_{j=1}^{t_m}w_{mj}-(t_m-1)]-1\\
&=&(\sum\limits_{x\in V(D\setminus V_m)}w(x)+\sum\limits_{j=1}^{t_m}w_{mj})-(|V(D\setminus V_m)|+t_m)+2\\
&=&\sum\limits_{x\in V(D)}w(x)-|V(D)|+2.
\end{eqnarray*}

(2) Next we will prove $\mbox{reg}\,((J_{i}:x_{m,\,i+1}^{w_{m,i+1}})(-w_{m,\,i+1}))\leq \mbox{reg}\,( J_{t_m})$ for $0\leq i\leq t_m-1$, where  $J_0=I(D)$. Thus the assertion  follows from Lemma \ref{lem2} (2) and by repeatedly using Lemma \ref{lem7} (1) on the short exact sequences $(\ddagger)$.

In fact, we can write $(J_{i}:x_{m,\,i+1}^{w_{m,i+1}})$ as
\begin{eqnarray*}
(J_{i}:x_{m,\,i+1}^{w_{m,i+1}})&=&I(D\setminus (V_{m-1}\sqcup V_m))+(x_{m-1,1},x_{m-1,2},\ldots,x_{m-1,\,t_{m-1}})\\
&+&(x_{m1}^{w_{m1}},x_{m2}^{w_{m2}},\ldots,x_{m,\,i}^{w_{m,i}})\\
&=&L_1+L_2+L_3^{i}
\end{eqnarray*}
where $L_1$ is the edge ideal of the  induced subgraph $D\setminus (V_{m-1}\sqcup V_m)$ of $D$  on the set $V\setminus (V_{m-1}\sqcup V_m)$ obtained
by removing the vertices in $V_{m-1}\sqcup V_m$ and the edges incident to these vertices, $L_2=(x_{m-1,1},x_{m-1,2},\ldots,x_{m-1,\,t_{m-1}})$,  $L_3^{0}=(0)$, and $L_3^{i}=(x_{m1}^{w_{m1}},x_{m2}^{w_{m2}},\ldots,x_{m,\,i}^{w_{m,i}})$ for $1\leq i\leq t_m-1$.
In fact, the variables  appearing in $L_1$, $L_2$ and $L_3^i$ are different from each other and $L_3^0=(0)$ for any  $0\leq i\leq t_m-1$.
We  distinguish into  the following two cases:

 (I) If $m=3$, then $L_1=(0)$. Thus, Lemma \ref{lem4} (2),  we have
\begin{eqnarray*}
\mbox{reg}\,((J_{0}:x_{31}^{w_{31}})(-w_{31}))&=&\mbox{reg}\,((J_{0}:x_{31}^{w_{31}})))+w_{31}\\
&=&\sum\limits_{i=1}^{t_{2}}\mbox{reg}\,(x_{2,i})-(t_{2}-1)+w_{31}=w_{31}+1\\
&=&(t_1+\sum\limits_{j=1}^{t_2}w_{2j}
+\sum\limits_{j=1}^{t_3}w_{3j})-(t_1+\sum\limits_{j=1}^{t_2}w_{2j}+\sum\limits_{j=2}^{t_3}w_{3j})
+1\\
&=&[\sum\limits_{x\in V(D)}w(x)-|V(D)|+2]+(t_1+t_2+t_3)\\
&-&(1+t_1+\sum\limits_{j=1}^{t_2}w_{2j}+\sum\limits_{j=2}^{t_3}w_{3j})\\
&\leq&\mbox{reg}\,( J_{t_m}),
\end{eqnarray*}
and, for $1\leq i\leq t_3-1$,
\begin{eqnarray*}
\mbox{reg}\,((J_{i}:x_{3,i+1}^{w_{3,i+1}})(-w_{3,i+1}))&=&\mbox{reg}\,((J_{i}:x_{3,\,i+1}^{w_{3,i+1}}))+w_{3,i+1}=\mbox{reg}\,(L_2+L_3^{i})+w_{3,i+1}\\
&=&[\mbox{reg}\,(L_2)+\mbox{reg}\,(L_3^{i})-1]+w_{3,i+1}\\
&=&1+[\sum\limits_{j=1}^{i}w_{3j}-(i-1)]-1+w_{3,i+1}\\
&=&[\sum\limits_{x\in V(D)}w(x)-|V(D)|+2]+(t_1+t_2+t_3-i)-\\
& &(1+t_1+\sum\limits_{j=1}^{t_2}w_{2j}+\sum\limits_{j=i+2}^{t_3}w_{3j})\\
&\leq&\mbox{reg}\,( J_{t_m})
\end{eqnarray*}
where the above  inequalities  hold because of $w_{2j}\geq 2$ for $1\leq j\leq t_2$.

(II) If $m\geq 4$. By induction hypothesis on $m$ and Lemma \ref{lem4} (2) and similarly arguments at above, we have
\begin{eqnarray*}
\mbox{reg}\,((J_{0}:x_{m1}^{w_{m1}})(-w_{m1}))\!\!\!&=&\!\!\!\mbox{reg}\,(L_1+L_2)+w_{m1}=
[\mbox{reg}\,(L_1)+\mbox{reg}\,(L_2)-1]+w_{m1}\\
&=&\!\!\!\!\!\sum\limits_{x\in V(D\setminus (V_{m-1}\sqcup V_m))}w(x)\!-\!|V(D\setminus (V_{m-1}\sqcup V_m))|\!+\!2\!+\!w_{m1}\\
&=&\!\![\!\!\!\sum\limits_{x\in V(D)}\!\!\!w(x)\!-\!|V(D)|\!+\!2]\!+\!t_{m-1}\!+\!t_{m}- \!\!\sum\limits_{j=1}^{t_{m-1}}\!w_{m-1,\,j}-\!\!\sum\limits_{j=2}^{t_m}\!w_{mj}\\
&\leq&\mbox{reg}\,( J_{t_m})
\end{eqnarray*}
and, for $1\leq i\leq t_m-1$,
\begin{eqnarray*}
\mbox{reg}\,((J_{i}:x_{m,\,i+1}^{w_{m,\,i+1}})(-w_{m,i+1}))&=&\mbox{reg}\,(L_1)+\mbox{reg}\,(L_2)+\mbox{reg}\,(L_3^{i})-2+w_{m,i+1}\\
&=&\!\!\!\!\!\sum\limits_{x\in V(D\setminus (V_{m-1}\sqcup V_m))}\!\!\!w(x)-|V(D\setminus (V_{m-1}\sqcup V_m))|+2\\
&+&1+\sum\limits_{j=1}^{i}w_{mj}-(i-1)-2+w_{m,i+1}\\
&=&[\sum\limits_{x\in V(D)}w(x)-|V(D)|+2]+(t_{m-1}+t_m-i)\\
&-&(\sum\limits_{j=1}^{t_{m-1}}w_{m-1,\,j}+\sum\limits_{j=i+2}^{t_m}w_{mj})\\
&\leq&\mbox{reg}\,( J_{t_m})
\end{eqnarray*}
where the above  inequalities  hold because of $w_{2j}\geq 2$ for $1\leq j\leq t_2$.
This completes the proof.
\end{proof}

\medskip
\begin{Theorem}\label{thm2}
Let $D=(V,E,w)$ be a  vertex-weighted   oriented graph  as Theorem \ref{thm1}. Then
 $$\mbox{pd}\,(I(D))=|V(D)|-2.$$
\end{Theorem}
\begin{proof}
Let $V_{i}=\{x_{i1},\ldots,x_{i,t_{i}}\}$ for $1\leq i\leq m$, then
\begin{eqnarray*}
I(D)&=&(x_{11}x_{21}^{w_{21}}\!,x_{11}x_{22}^{w_{22}}\!,\ldots,x_{11}x_{2,\,t_2}^{w_{2,t_2}}\!, x_{12}x_{21}^{w_{21}}\!,\ldots,x_{12}x_{2,\,t_2}^{w_{2,t_2}}\!,\ldots, x_{1,\,t_1}x_{21}^{w_{21}}\!,\ldots,\\
& &x_{1,\,t_1}x_{2,t_2}^{w_{2,\,t_2}}\!,x_{21}x_{31}^{w_{31}}\!,\ldots, x_{21}x_{3,\,t_3}^{w_{3,t_3}}\!, x_{22}x_{31}^{w_{31}}\!,\ldots,
x_{22}x_{3,\,t_3}^{w_{3,t_3}}\!,\ldots,x_{2,\,t_2}x_{31}^{w_{31}}\!,\ldots, \\
& &x_{2,\,t_2}x_{3,\,t_3}^{w_{3,t_3}}\!,\ldots, x_{m-1,1}x_{m1}^{w_{m1}}\!,\ldots, x_{m-1,1}x_{m,\,t_m}^{w_{m,t_m}}\!,x_{m-1,2}x_{m1}^{w_{m1}}\!,\ldots, x_{m-1,2}x_{m,\,t_m}^{w_{m,t_m}}\!,\\
& &x_{m-1,\,t_{m-1}}x_{m1}^{w_{m1}}\!,\ldots,x_{m-1,\,t_{m-1}}x_{m,\,t_m}^{w_{m,t_m}}).
\end{eqnarray*}
We apply induction on $m$. The case $m=2$ follows from Lemma \ref{lem8} (2).
Now we assume that $m\geq3$. Consider the following short exact sequences

\begin{gather*}
\begin{matrix}
 0 & \longrightarrow & \frac{S}{(I(D)\,:\,x_{m-1,1})}(-1) & \stackrel{ \cdot x_{m-1,1}} \longrightarrow & \frac{S}{I(D)} &\longrightarrow & \frac{S}{J_1} & \longrightarrow & 0 \\
 \\
 0 & \longrightarrow & \frac{S}{(J_1\,:\,x_{m-1,2})}(-1) & \stackrel{ \cdot x_{m-1,2}} \longrightarrow  & \frac{S}{J_1} & \longrightarrow & \frac{S}{J_2} & \longrightarrow & 0 &\hspace{1.0cm} (\ddagger\ddagger) \\
 \\
     &  &\vdots&  &\vdots&  &\vdots&  &\\
     \\
 0&  \longrightarrow & \frac{S}{(J_{t_{m-1}-1}\,:\,x_{m-1,t_{m-1}})}(-1)& \stackrel{ \cdot x_{{m-1},t_{m-1}}} \longrightarrow & \frac{S}{J_{t_{m-1}-1}}& \longrightarrow & \frac{S}{J_{t_{m-1}}}& \longrightarrow & 0\\
 \end{matrix}\quad
\end{gather*}
where $J_{i}=I(D)+(x_{m-1,1},\ldots,x_{m-1,i})$ for $1\leq i\leq t_{m-1}$. We  prove this argument in the following two steps.

 (1) We first prove   $\mbox{pd}\,(J_{i}:x_{m-1,\,i+1})=|V(D)|-2$\ \  for $0\leq i\leq t_{m-1}-1$, where $J_0=I(D)$.
We  write $(J_{i}:x_{m-1,i+1})$ as follows:
\[
(J_{i}:x_{m-1,\,i+1})=L_{1}^{i}+L_{2}^{i}+L_{3},
\]
where $L_{1}^{0}=(0)$, $L_{1}^{i}=(x_{m-1,1},x_{m-1,2},\ldots,x_{m-1,\,i})$ for $1\leq i\leq t_{m-1}-1$,
$L_{2}^{i}=I(D\setminus (V_{m-1}\sqcup V_{m}))+(x_{m-2,\,1}x_{m-1,\,i+1}^{w_{m-1,\,i+1}-1}, x_{m-2,\,1}x_{m-1,\,i+2}^{w_{m-1,i+2}},\ldots, x_{m-2,\,1}x_{m-1,\,t_{m-1}}^{w_{m-1,t_{m-1}}}, \\
x_{m-2,\,2}x_{m-1,\,i+1}^{w_{m-1,i+1}-1},x_{m-2,\,2}x_{m-1,\,i+2}^{w_{m-1,i+2}},\ldots,
x_{m-2,\,2}x_{m-1,\,t_{m-1}}^{w_{m-1,t_{m-1}}},\ldots,x_{m-2,\,t_{m-2}}x_{m-1,\,i+1}^{w_{m-1,\,i+1}-1},\\ x_{m-2,\,t_{m-2}}x_{m-1,\,i+2}^{w_{m-1,i+2}}, \ldots, x_{m-2,\,t_{m-2}}x_{m-1,\,t_{m-1}}^{w_{m-1,t_{m-1}}})$ for $0\leq i\leq t_{m-1}-1$ and  $L_{3}=(x_{m1}^{w_{m1}}\!\!,\\
x_{m2}^{w_{m2}}\!, \ldots,x_{m,\,t_m}^{w_{m,t_m}})$.
In fact, $L_{2}^{i}$ is an induced subgraph $H_i$ of $D$ on the set $(\mathop\bigsqcup\limits_{j=1}^{m-2}V_j)\bigsqcup\\
\{x_{m-1,\,i+1},\ldots,x_{m-1,\,t_{m-1}}\}$ with the weight function  $w_{i}': V(H_i)\rightarrow \mathbb{N}^{+}$
such that $w_{i}'(x_{m-1,\,i+1})
=w_{m-1,\,i+1}-1$ and $w_{i}'(x)=w(x)$ for  any other vertex $x\in V(H_i)$.
Therefore,  by induction hypothesis on $m$ , we obtain
\[
\mbox{pd}\,(L_{2}^{i})=|V(D_i)|-2=|V(D)|-t_m-i-2  \ \ \text{for\ any}\  \  0\leq i\leq t_{m-1}-1.
\]

 Next we will compute $\mbox{pd}\,(J_{i}:x_{m-1,\,i+1})$. Note that for $0\leq i\leq t_{m-1}-1$
 \[
 (J_{i}:x_{m-1,\,i+1})=L_{1}^{i}+L_{2}^{i}+L_{3}.
\]

 Since $L_{1}^{0}=(0)$, $L_{1}^{i}=(x_{m-1,1},x_{m-1,2},\ldots,x_{m-1,\,i})$ for $1\leq i\leq t_{m-1}-1$, and the variables appearing in  $L_1^i$, $L_2^i$ and $L_3$ are different from each other,
  we obtain that from Lemma \ref{lem4} (1),
\begin{eqnarray*}
\mbox{pd}\,(J_{0}:x_{m-1,1})&=&\mbox{pd}\,(L_{2}^{0})+\mbox{pd}\,(L_{3})+1\\
&=&(|V(D)|-t_m-2)+(t_{m}-1)+1\\
&=&|V(D)|-2,\\
\mbox{pd}\,(J_{i}:x_{m-1,\,i+1})&=&\mbox{pd}\,(L_{1}^{i})+\mbox{pd}\,(L_{2}^{i})+\mbox{pd}\,(L_{3})+2\\
&=&(i-1)+(|V(D)|-t_m-i-2)+(t_{m}-1)+2\\
&=&|V(D)|-2.
\end{eqnarray*}

(2) Next we will prove  $\mbox{pd}\,( J_{t_{m-1}})\leq |V(D)|-3$, this implies that $\mbox{pd}\,(J_{t_{m-1}})< \mbox{pd}\,(J_{i}:x_{m-1,\,i+1})\ \ \mbox{for all}\ \ 0\leq i \leq t_{m-1}-1$.  Thus the assertion  follows from Lemma \ref{lem2} (1) and by repeatedly using Lemma \ref{lem7} (2) on the short exact sequences $(\ddagger\ddagger)$.

First, we notice that $J_{t_{m-1}}=I(D\setminus (V_{m-1}\sqcup V_m))+(x_{m-1,1},x_{m-1,2},\ldots,x_{m-1,t_{m-1}})$. We consider the following two cases:

(I) If $m=3$, then $I(D\setminus (V_{2}\sqcup V_3))=(0)$. Hence
$$\mbox{pd}\,( J_{t_{2}})=\mbox{pd}\,((x_{21},x_{22},\ldots,x_{2,t_2}))=t_2-1=|V(D)|-t_1-t_3-1\leq |V(D)|-3.$$

(II) If $m\geq4$. Since all the generators of  $I(D\setminus (V_{m-1}\sqcup V_m)$ can not divided by variables $x_{m-1,\,i}$  where $1\leq i \leq t_{m-1}$, we have
\begin{eqnarray*}
\mbox{pd}\,( J_{t_{m-1}})&=&\mbox{pd}\,(I(D\setminus (V_{m-1}\sqcup V_m)))+\mbox{pd}\,((x_{m-1,1},x_{m-1,2},\ldots,x_{m-1,\,t_{m-1}}))+1\\
&=&[|V(D\setminus (V_{m-1}\sqcup V_{m}))|-2]+(t_{m-1}-1)+1\\
&=&|V(D)|-2-t_m\\
&\leq& |V(D)|-3.
\end{eqnarray*}
The proof is complete.
\end{proof}

\medskip
An immediate consequence of the above theorem is the following corollary.
\begin{Corollary}\label{cor3}
Let $D=(V(D),E(D),w)$ be a vertex-weighted  oriented graph as Theorem \ref{thm1}. Then
$$\mbox{depth}\,(I(D))=2$$
\end{Corollary}
\begin{proof}
It follows from Auslander-Buchsbaum formula and the above theorem.
\end{proof}

\medskip
The following   example shows that the projective dimension and the
regularity  of the edge ideals of vertex-weighted  oriented graphs as Theorem \ref{thm1} are related to direction selection.
\begin{Example}  \label{example2}
Let $I(D)=(x_1x_2^{2},x_2x_3^{2},x_4x_2^{2},x_3x_5^{2},x_4x_5^{2},x_6x_5^{2},x_6x_7^{2})$ be the edge ideal of vertex-weighted oriented $6$-partite graph $D=(V,E,w)$ with $w_1=w_4=w_6=1$ and  $w_2=w_3=w_5=w_7=2$, where  $V=\mathop\bigsqcup\limits_{j=1}^{6}V_j$ with $V_{1}=\{x_1\}$, $V_{2}=\{x_2\}$, $V_{3}=\{x_3,x_4\}$, $V_{4}=\{x_5\}$, $V_{5}=\{x_6\}$ and $V_{6}=\{x_7\}$. By using CoCoA, we obtain $\mbox{reg}\,(I(D))=7$ and $\mbox{pd}\,(I(D))=4$. But we have
$\mbox{reg}\,(I(D))=\sum\limits_{i=1}^{7}w_i-|V(D)|+2=6$
 by Theorem \ref{thm1}
and $\mbox{pd}\,(I(D))=|V(D)|-2=5$ by Theorem \ref{thm2}.
 \end{Example}

\medskip
The following example shows that the assumption  that  $w(x)\geq 2$ if  $x\in V\setminus (V_1\sqcup V _m)$ in Theorem \ref{thm1} and  Theorem \ref{thm2} cannot be dropped.
\begin{Example}  \label{example3}
Let $I(D)=(x_1x_3^{2},x_1x_4^{2},x_2x_3^{2},x_2x_4^{2},x_3x_5,x_3x_6,x_4x_5,x_4x_6,x_5x_7^{2},x_6x_7^{2})$ be the edge ideal of vertex-weighted $4$-partite  digraph $D=(V,E,w)$ with $w_1=w_2=w_5=w_6=1$ and  $w_3=w_4=w_7=2$, where $V=\mathop\bigsqcup\limits_{j=1}^{4}V_j$ with $V_{1}=\{x_1,x_2\}$, $V_{2}=\{x_3,x_4\}$, $V_{3}=\{x_5,x_6\}$  and $V_{4}=\{x_7\}$. By using CoCoA, we obtain $\mbox{reg}\,(I(D))=4$ and  $\mbox{pd}\,(I(D))=4$. But we have
$\mbox{reg}\,(I(D))=\sum\limits_{i=1}^{7}w_i-|V(D)|+2=5$
 by Theorem \ref{thm1}
and $\mbox{pd}\,(I(D))=|V(D)|-2=5$ by Theorem \ref{thm2}.
\end{Example}

\medskip

\section{Projective dimension and regularity of  edge ideals of the second  class of  vertex-weighted oriented $m$-partite  graphs}

In this section, we will provide some exact formulas for the projective
dimension and the regularity of  the edge ideals of  some vertex-weighted oriented  $m$-partite  graphs with   whiskers. Such graphs are
another class of  vertex-weighted  oriented $m$-partite  graphs with  vertex set $V=\mathop\bigsqcup\limits_{i=1}^{m}V_i$ with $|V_1|\leq |V_2|$, the  edge set $E=\bigcup\limits_{i=1}^{m-1}E(D_{i})$,  where  $D_1$ is a bipartite graph with the vertex set $\{x_{11},\ldots,x_{1,|V_1|}\}\sqcup \{x_{21},\ldots,x_{2,|V_1|}\}$, the edge set $\{x_{21}x_{11},\ldots,x_{2,|V_1|}x_{1,|V_1|}\}$ and
$D_i$ is  a complete bipartite graph and it is also an induced subgraph  of $D$ on $V_{i}\sqcup V_{i+1}$ satisfying the starting point of  every edge of $E(D_{i})$ belongs to $V_{i}$ and  its ending point belongs to $V_{i+1}$ for $2\leq i\leq m-1$.
 We also give some examples to show that these formulas  are related to direction selection
and the weight of vertices.

\medskip
Now we are ready to present the main theorem of this section.
\begin{Theorem}\label{thm3}Let $m\geq 2$ be an integer, and suppose that $D=(V,E,w)$ is a  vertex-weighted oriented  $m$-partite  graph, its  vertex set $V=\mathop\bigsqcup\limits_{i=1}^{m}V_i$ with $|V_1|\leq |V_2|$, its edge set $E=\bigcup\limits_{i=1}^{m-1}E(D_{i})$,  where  $D_1$ is a bipartite graph with the vertex set $\{x_{11},\ldots,x_{1,|V_1|}\}\sqcup \{x_{21},\ldots,x_{2,|V_2|}\}$, the edge set $\{x_{21}x_{11},\ldots,x_{2,|V_1|}x_{1,|V_1|}\}$ and
$D_i$ is  a complete bipartite graph and it is also an induced subgraph  of $D$ on $V_{i}\sqcup V_{i+1}$ satisfying the starting point of  every edge in $E(D_{i})$ belongs to $V_{i}$ and  its ending point belongs to $V_{i+1}$ for $2\leq i\leq m-1$. If $w(x)\geq 2$ for any  $x\in V\setminus (V_1\sqcup V _m)$. Then
\begin{itemize}
\item[(1)] $\mbox{reg}\,(I(D))=\sum\limits_{x\in V(D)}w(x)-|V(D\setminus V_1)|+1$,
\item[(2)] $\mbox{pd}\,(I(D))=\left\{\begin{array}{ll}
|V(D\setminus V_1)|-2,\ \ &\text{if}\ \  |V_1|<|V_2|,\\
|V(D\setminus V_2)|-1,\ \ &
\text{if}\ \   |V_1|=|V_2|.
\end{array}\right.$
\end{itemize}
\end{Theorem}
\begin{proof} Let $V_{i}=\{x_{i1},\ldots,x_{i,\,t_{i}}\}$ for $1\leq i\leq m$, then
\begin{eqnarray*}
I(D)&=&(x_{21}x_{11}^{w_{11}}\!, x_{22}x_{12}^{w_{12}}\!,\ldots,
x_{2,\,t_1}x_{1,\,t_1}^{w_{1,t_1}}\!,x_{21}x_{31}^{w_{31}}\!, x_{21}x_{32}^{w_{32}}\!,\ldots, x_{21}x_{3,\,t_3}^{w_{3,t_3}}\!, x_{22}x_{31}^{w_{31}}\!,\\
& &\ldots,x_{22}x_{3,\,t_3}^{w_{3,t_3}}\!,\ldots,x_{2,\,t_2}x_{31}^{w_{31}}\!,\ldots,
x_{2,\,t_2}x_{3,\,t_3}^{w_{3,t_3}}\!,\ldots,x_{m-1,1}x_{m1}^{w_{m1}}\!,\ldots, x_{m-1,1}x_{m,\,t_m}^{w_{m,t_m}}\!, \\
& &x_{m-1,2}x_{m2}^{w_{m2}}\!,\ldots, x_{m-1,2}x_{m,\,t_m}^{w_{m,t_m}}\!,x_{m-1,\,t_{m-1}}x_{m1}^{w_{m1}}\!,\ldots,x_{m-1,\,t_{m-1}}x_{m,\,t_m}^{w_{m,t_m}}).
\end{eqnarray*}
We apply induction on $m$. If $m=2$. Since the underlying graph $G$ of $D$ is simple, it has not isolated vertex. Thus $t_1=t_2$ and
$I(D)=(x_{21}x_{11}^{w_{11}}\!, x_{22}x_{12}^{w_{12}}\!,\ldots,
x_{2,\,t_1}x_{1,\,t_1}^{w_{1,t_1}})$.
In this case, by Lemma \ref{lem4}, we have
\begin{eqnarray*}
\mbox{reg}\,(I(D))&=&\sum\limits_{j=1}^{t_1}(1+w_{1j})-(t_1-1)= \sum\limits_{x\in V(D)}w(x)-|V(D\setminus V_1)|+1,\\
\mbox{pd}\,(I(D))&=&\sum\limits_{j=1}^{t_1}\mbox{pd}\,((x_{2j}x_{1j}^{w_{1j}}))+(t_1-1)=t_1-1=|V(D\setminus V_2)|-1.
\end{eqnarray*}

Now we assume that $m\geq3$. For $0\leq i\leq t_1-2$,  we set  $J_{0}=I(D)$, $K_{i+1}=(x_{2,\,i+1}x_{1,\,i+1}^{w_{1,i+1}})$,
$L_{i+1}=(x_{2,\,i+2}x_{1,\,i+2}^{w_{1,i+2}}, \ldots, x_{2,\,t_1}x_{1,\,t_1}^{w_{1,t_1}})$, $K_{t_1}=(x_{2,\,t_1}x_{1,\,t_1}^{w_{1,t_1}})$
 and $L_{t_1}=(0)$.
Further, we assume that $J_{i+1}=L_{i+1}+I(D\setminus V_1)$ and  $J_{t_1}=I(D\setminus V_1)$.
Thus,  for all  $0\leq i \leq t_{1}-1$, we have
$$\hspace{1.5cm}J_{i}=J_{i+1}+K_{i+1}\ \ \ \  \text{and}\ \ \ \  J_{i+1}\cap K_{i+1}=K_{i+1}(L_{i+1}+L)$$
where  $L=(x_{31}^{w_{31}},x_{32}^{w_{32}},\ldots,x_{3,\,t_3}^{w_{3,t_3}}) +I(D\setminus (V_1\sqcup V_2))$.

For any $0\leq i \leq t_{1}-1$, because the variable $x_{1,\,{i+1}}^{w_{1,i+1}}$ in $K_{i+1}$  can not divided the generators of $J_{i+1}$ and   $K_{i+1}$ has  a linear resolution,  it follows that
$J_{i}=J_{i+1}+K_{i+1}$ is a Betti splitting. By Corollary \ref{cor1}, we obtain
\begin{eqnarray*}
\mbox{reg}\,(J_{i})&=&\mbox{max}\{\mbox{reg}\,(K_{i+1}),\mbox{reg}\,(J_{i+1}), \mbox{reg}\,(K_{i+1}\cap J_{i+1})-1\},\\
\mbox{pd}\,(J_{i})&=&\mbox{max}\{\mbox{pd}\,(K_{i+1}),\mbox{pd}\,(J_{i+1}), \mbox{pd}\,(K_{i+1}\cap J_{i+1})+1\}.
\end{eqnarray*}
Note that  the variables appearing in $L_{i+1}$, $K_{i+1}$ and $L$  are different from each other. Repeated using  the above formulas, Lemmas \ref{lem4} and \ref{lem5}, we obtain
\begin{eqnarray*}
( * )\hspace{0.5cm} \mbox{reg}\,(J_{0})\!\!\!&=&\!\!\!\mbox{max}\{\mbox{reg}\,(K_{i+1}),\mbox{reg}\,(J_{t_1}), \mbox{reg}\,(K_{i+1}\!\cap\! J_{i+1})\!-\!1,
 \  \text{for } \ 0\leq i \leq t_{1}-1\}\\
&=&\!\!\!\mbox{max}\{\mbox{reg}\,(K_{i+1}),\mbox{reg}\,(J_{t_1}), \mbox{reg}\,(K_{i+1})+\mbox{reg}\,(L_{i+1}+L)-1,\\
 & & \ \ \ \ \ \ \  \text{for } \ 0\leq i \leq t_{1}-1\}\\
 \!\!\!&=&\!\!\!
\mbox{max}\{\mbox{reg}\,(K_{i+1}),\mbox{reg}\,(K_{t_1}), \mbox{reg}\,(K_{i+1})+\mbox{reg}\,(L_{i+1})+\mbox{reg}\,(L)-2, \\
& & \mbox{reg}\,(J_{t_1}), \mbox{reg}\,(K_{t_1})+\mbox{reg}\,(L)-1,  \ \ \  \text{for} \  0\leq i \leq t_{1}-2 \}
\end{eqnarray*}
and
\begin{eqnarray*}
( ** )\hspace{0.5cm}\mbox{pd}\,(J_{0})\!\!\!&=&\!\!\!\mbox{max}\{\mbox{pd}\,(K_{i+1}),\mbox{pd}\,(J_{t_1}), \mbox{pd}\,(K_{i+1}\cap J_{i+1})\!+\!1,
 \  \text{for } \ 0\leq i \leq t_{1}\!-\!1\}\\
&=&\!\!\!\mbox{max}\{\mbox{pd}\,(K_{i+1}),\mbox{pd}\,(J_{t_1}),\mbox{pd}\,(L_{i+1}+L)+1, \ \  \text{for } \ 0\leq i \leq t_{1}-1\}\\
\!\!\!&=&\!\!\!
\mbox{max}\{\mbox{pd}\,(J_{t_1}), \mbox{pd}\,(L_{i+1})\!+\!\mbox{pd}\,(L)\!+\!2,\mbox{pd}\,(L)\!+\!1,\  \text{for} \  0\leq i \leq t_{1}\!-\!2 \}.
\end{eqnarray*}
Next, we will prove that $\mbox{reg }\,(L)=\sum\limits_{\ell=3}^{m}(\sum\limits_{j=1}^{t_{\ell}}w_{\ell,\,j} )-\sum\limits_{j=3}^{m}t_{j}+1$ and
 $\mbox{pd}\,(L)=\sum\limits_{j=3}^{m}t_{j}-1$.

We consider the following two cases:

(I)  If $m=3$, then $I(D\setminus (V_1\sqcup V_2))=(0)$. This implies that $L=(x_{31}^{w_{31}},x_{32}^{w_{32}},\ldots,x_{3,\,t_3}^{w_{3,t_3}}) $. It follows that from Lemma \ref{lem4}
\begin{eqnarray*}\mbox{reg}\,(L)&=&\sum\limits_{j=1}^{t_{3}}w_{3j}-(t_3-1)=\sum\limits_{j=1}^{t_{3}}w_{3j}-t_3+1,\\
\mbox{pd}\,(L)&=&t_3-1.
\end{eqnarray*}

(II) If $m\geq 4$, then the polarization $L^{\mathcal {P}}$ of the ideal $L$  can be regarded as the  polarization of the edge ideal of a vertex-weighted oriented graph $H$ with whiskers, its vertex set $(\mathop\bigsqcup\limits_{j=3}^{m}V_j)\sqcup\{y_{31},\ldots,y_{3,\,t_3}\}$, edge set
$E(D\setminus (V_1\sqcup V_2))\cup \{x_{31}y_{31},\ldots, x_{3,\,t_3}y_{3,\,t_3}\}$, and its weight function  is $w': V(H)\rightarrow \mathbb{N}^{+}$ with $w'(x_{3j})=1$, $w'(y_{3j})=w_{3j}-1$ for $1\leq j\leq t_3$ and $w'(x)=w(x)$ for  any other vertex $x$.  It follows that  $w'(x_{3j})+w'(y_{3j})=w_{3j}$
for $1\leq j\leq t_3$.
Notice that $H$ has only  $(m-1)$-partition.
By induction hypothesis and Lemma \ref{lem6} we obtain
\begin{eqnarray*}
\mbox{reg}\,(L)&=&\sum\limits_{x\in V(H)}w'(x)-|V(H\setminus V_3)|+1=\sum\limits_{\ell=3}^{m}(\sum\limits_{j=1}^{t_{\ell}}w_{\ell,\,j} )-\sum\limits_{j=3}^{m}t_{j}+1,\\
\mbox{pd}\,(L)&=&|V(H\setminus V_3)|-1=\sum\limits_{j=3}^{m}t_{j}-1.
\end{eqnarray*}

Therefore, from the formulas $( * )$ and $( ** )$, we have
\begin{eqnarray*}
\mbox{reg}\,(J_{0})\!\!\!&=&\!\!\!\mbox{max}\{\mbox{reg}\,(K_{i+1}),\mbox{reg}\,(K_{t_1}), \mbox{reg}\,(K_{i+1})+\mbox{reg}\,(L_{i+1})+\mbox{reg}\,(L)-2, \\
& & \mbox{reg}\,(J_{t_1}), \mbox{reg}\,(K_{t_1})+\mbox{reg}\,(L)-1,  \ \ \  \text{for} \  0\leq i \leq t_{1}-2 \}\\
&=&
\mbox{max}\{1+w_{1,\,i+1},1+w_{1,\,t_1},(1+w_{1,\,i+1})+[\sum\limits_{j=i+2}^{t_1}(1+w_{1j})-(t_1-i-2)] \\
\!\!\!&+&\!\!\!(\sum\limits_{\ell=3}^{m}(\sum\limits_{j=1}^{t_{\ell}}w_{\ell,\,j} )-\sum\limits_{j=3}^{m}t_{j}+1)-2,\sum\limits_{x\in V(D\setminus V_1)}w(x)-|V(D\setminus V_1)|+2,\\
\!\!\!& &\!\!\!(1+w_{1,\,t_1})+(\sum\limits_{\ell=3}^{m}(\sum\limits_{j=1}^{t_{\ell}}w_{\ell,\,j} )-\sum\limits_{j=3}^{m}t_{j}+1)-1,  \ \ \  \text{for} \  0\leq i \leq t_{1}-2\}\\
\!\!\!&=&\!\!\!\sum\limits_{x\in V(D)}w(x)-|V(D\setminus V_1)|+1
\end{eqnarray*}
where this  maximal value is obtained when $i=0$, and
\begin{eqnarray*}
\mbox{pd}\,(J_{0})\!\!\!&=&\!\!\!\mbox{max}\{\mbox{pd}\,(J_{t_1}), \mbox{pd}\,(L_{i+1})\!+\!\mbox{pd}\,(L)\!+\!2,\mbox{pd}\,(L)\!+\!1,\  \text{for} \  0\leq i \leq t_{1}\!-\!2 \}\\
&=&
\mbox{max}\{|V(D\setminus V_1)|-2, (t_1-i-2)+(\sum\limits_{j=3}^{m}t_{j}-1)+2,(\sum\limits_{j=3}^{m}t_{j}-1)+1 \}\\
\!\!\!&=&\!\!\!\mbox{max}\{\sum\limits_{j=2}^{m}t_{j}-2,(\sum\limits_{j=2}^{m}t_{j}-2)+(1+t_1-t_2) \}\\
\!\!\!&=&\!\!\!\left\{\begin{array}{ll}
|V(D\setminus V_1)|-2\ \ &\text{if}\  \  t_1<t_2,\\
|V(D\setminus V_2)|-1\ \ &
\text{if}\ \  t_1=t_2.
\end{array}\right.
\end{eqnarray*}
The proof is complete.
\end{proof}

\medskip
An immediate consequence of the above theorem is the following corollary.
\begin{Corollary}\label{cor4}
Let $D=(V(D),E(D),w)$ be a vertex-weighted  oriented graph as Theorem \ref{thm3}. Then
$$\mbox{depth}\,(I(D))=\left\{\begin{array}{ll}
|V_1|-2\ \ &\text{if}\  \  t_1<t_2,\\
|V_2|-1\ \ &
\text{if}\ \  t_1=t_2.
\end{array}\right.$$
\end{Corollary}
\begin{proof}
It follows from Auslander-Buchsbaum formula and the above theorem.
\end{proof}
 The following  example shows that the projective dimension and the
regularity  of the edge ideals of vertex-weighted oriented  graphs as Theorem \ref{thm3} are related to direction selection.
\begin{Example}  \label{example6}
Let $I(D)=(x_2x_1,x_2x_3^{3},x_2x_4^{3},x_3x_5^{3},x_4x_5^{3},x_6x_5^{3},x_6x_7^{3})$ be the edge ideal of vertex-weighted oriented $6$-partite graph $D=(V,E,w)$ with $w_1=w_2=w_6=1$  and  $w_3=w_4=w_5=w_7=3$,  where $V=\mathop\bigsqcup\limits_{j=1}^{6}V_j$ with  $V_{1}=\{x_1\}$, $V_{2}=\{x_2\}$, $V_{3}=\{x_3,x_4\}$, $V_{4}=\{x_5\}$, $V_{5}=\{x_6\}$ and $V_{6}=\{x_7\}$. By using CoCoA, we obtain $\mbox{reg}\,(I(D))=11$ and $\mbox{pd}\,(I(D))=4$. But we have $\mbox{reg}\,(I(D))=\sum\limits_{i=1}^{7}w_i-|V(D\setminus V_1)|+1=10$
and $\mbox{pd}\,(I(D))=|V(D\setminus V_2)|-1=5$
 by Theorem \ref{thm3}.
 \end{Example}

\medskip
The following example shows that the assumption  that  $w(x)\geq 2$ if $x\in V\setminus(V_1\sqcup V_m)$  in Theorem \ref{thm3} cannot be dropped.
\begin{Example}  \label{example7}
Let $I(D)=(x_3x_1^{3},x_4x_2^{3},x_3x_5^{3},x_3x_6^{3},x_4x_5^{3},x_4x_6^{3},x_5x_7,x_6x_7,x_7x_8^{3})$ be the edge ideal of vertex-weighted  whisker $D=(V,E,w)$ with $w_1=w_2=w_5=w_6=w_8=3$ and  $w_3=w_4=w_7=1$, where $V=\mathop\bigsqcup\limits_{j=1}^{5}V_j$ with $V_{1}=\{x_1,x_2\}$, $V_{2}=\{x_3,x_4\}$, $V_{3}=\{x_5,x_6\}$, $V_{4}=\{x_7\}$ and $V_{5}=\{x_8\}$. By using CoCoA, we obtain $\mbox{reg}\,(I(D))=11$ and $\mbox{pd}\,(I(D))=4$. But we have $\mbox{reg}\,(I(D))=\sum\limits_{i=1}^{8}w_i-|V(D\setminus V_1)|+1=13$
and  $\mbox{pd}\,(I(D))=|V(D\setminus V_2)|-1=5$
 by Theorem \ref{thm3}.
\end{Example}

\medskip

\section{Projective dimension and regularity of edge ideals of the third class of  vertex-weighted oriented $m$-partite graphs}

In this section, we will give some exact formulas for the projective
dimension and the regularity of  edge ideals of   the third  class of  vertex-weighted  oriented $m$-partite graphs  with vertex set $V=\mathop\bigsqcup\limits_{i=1}^{m}V_i$ and  edge set $E=\bigcup\limits_{i=1}^{m}E(D_{i})$, where $D_i$ is  a complete bipartite graph and it is also an induced subgraph  of $D$ on $V_{i}\sqcup V_{i+1}$ satisfying the starting point of  every edge of $E(D_{i})$ belongs to $V_{i}$ and  its ending point belongs to $V_{i+1}$ for $1\leq i\leq m$,  where we stipulate $V_{m+1}=V_1$. We also give some examples to show that these formulas  are related to direction selection
and the weight of vertices.

\begin{Theorem}\label{thm4}
Let $m\geq 3$ be an integer, and assume that $D=(V,E,w)$ is a  vertex-weighted oriented $m$-partite  graph with vertex set $V=\mathop\bigsqcup\limits_{i=1}^{m}V_i$ and   edge set $E=\bigcup\limits_{i=1}^{m}E(D_{i})$, where $D_i$ is  a complete bipartite graph and it is also an induced subgraph  of $D$ on $V_{i}\sqcup V_{i+1}$ satisfying the starting point of  every edge in $E(D_{i})$ belongs to $V_{i}$ and  its ending point belongs to $V_{i+1}$ for $1\leq i\leq m$,  where we stipulate $V_{m+1}=V_1$.
If $w(x)\geq 2$ for all $x\in V$. Then
$$\mbox{reg}\,(I(D))=\sum\limits_{x\in V(D)}w(x)-|V(D)|+1.$$
\end{Theorem}
\begin{proof}
Let $V_{i}=\{x_{i1},\ldots,x_{i,\,t_{i}}\}$ for $1\leq i\leq m$. Then
\begin{eqnarray*}
I(D)\!\!\!&=&\!\!\!(x_{11}x_{21}^{w_{21}}\!,x_{11}x_{22}^{w_{22}}\!,\ldots,x_{11}x_{2,\,t_2}^{w_{2,t_2}}\!, x_{12}x_{21}^{w_{21}}\!,\ldots,x_{12}x_{2,\,t_2}^{w_{2,t_2}}\!,\ldots,x_{1,\,t_1}x_{21}^{w_{21}}\!,\ldots, \\
& & x_{1,t_1}x_{2,\,t_2}^{w_{2,t_2}}\!,x_{21}x_{31}^{w_{31}}\!,\ldots, x_{21}x_{3,\,t_3}^{w_{3,t_3}}\!,x_{22}x_{31}^{w_{31}}\!,\ldots, x_{22}x_{3,\,t_3}^{w_{3,t_3}}\!,\ldots,  x_{2,\,t_2}x_{31}^{w_{31}}\!, \ldots,\\
& &x_{2,\,t_2}x_{3,\,t_3}^{w_{3,t_3}}\!,\ldots,x_{m-1,1}x_{m1}^{w_{m1}},\ldots,x_{m-1,1}x_{m,\,t_m}^{w_{m,t_m}}\!, x_{m-1,2}x_{m1}^{w_{m1}}\!,\ldots, x_{m-1,2}x_{m,\,t_m}^{w_{m,t_m}}\!,\\
& &\ldots,x_{m-1,\,t_{m-1}}x_{m1}^{w_{m1}}\!, \ldots,x_{m-1,\,t_{m-1}}x_{m,\,t_m}^{w_{m,t_m}}\!,x_{m1}x_{11}^{w_{11}}\!, \ldots, x_{m1}x_{1,\,t_1}^{w_{1,t_1}}\!,x_{m2}x_{11}^{w_{11}}\!, \\
& & \ldots,x_{m2}x_{1,\,t_1}^{w_{1,t_1}}\!,\ldots, x_{m,\,t_{m}}x_{11}^{w_{11}}\!,
\ldots,x_{m,\,t_{m}}x_{1,\,t_1}^{w_{1,t_1}}).
\end{eqnarray*}
Consider the following short exact sequences
\begin{gather*}
\begin{matrix}
 0 & \longrightarrow & \frac{S}{(I(D)\,:\,x_{m1}^{w_{m1}})}(-w_{m1}) & \stackrel{ \cdot x_{w_{m1}}} \longrightarrow & \frac{S}{I(D)} &\longrightarrow & \frac{S}{J_1} & \longrightarrow & 0 \\
     &  & &  & &  &  &\\
     0 & \longrightarrow & \frac{S}{(J_1\,:\,x_{m2}^{w_{m2}})}(-w_{m2}) & \stackrel{ \cdot x_{m2}^{w_{m2}}} \longrightarrow  &   \frac{S}{J_1} & \longrightarrow & \frac{S}{J_2} & \longrightarrow & 0  &\hspace{1.5cm} (\ddagger\ddagger\ddagger) \\
    &  & &  & &  &  &\\
     &  &\vdots&  &\vdots&  &\vdots&  &\\
 &  & &  & &  &  &\\
 0&  \longrightarrow & \frac{S}{(J_{t_m-1}\,:\,x_{m,\,t_m}^{w_{m,t_m}})}(-w_{m,\,t_m})& \stackrel{ \cdot x_{m,\,t_m}^{w_{m,t_m}}} \longrightarrow & \frac{S}{J_{t_m-1}}& \longrightarrow & \frac{S}{J_{t_m}}& \longrightarrow & 0\\
 \end{matrix}\quad
\end{gather*}
where $J_{i}=I(D)+(x_{m1}^{w_{m1}},\ldots,x_{m,\,i}^{w_{m,\,i}})$ for $1\leq i\leq t_m$. We  prove this argument in the following two steps.

\medskip
(1)  We first prove  $\mbox{reg}\,((J_{i}:x_{m,\,i+1}^{w_{m,i+1}}))\leq \sum\limits_{x\in V(D)}w(x)-|V(D)|+1-w_{m,\,i+1}$, for $0\leq i\leq t_m-1$, where $J_0=I(D)$.

In fact, for $0\leq i\leq t_m-1$, we can write $(J_{i}:x_{m,\,i+1}^{w_{m,i+1}})$ as
\[
(J_{i}:x_{m,\,i+1}^{w_{m,i+1}})=L_1+L_2+L^{i}
\]
where $L_{1}=(x_{11}^{w_{11}},x_{12}^{w_{12}},\ldots,x_{1,\,t_1}^{w_{1,t_1}}) +I(D\setminus (V_{m-1}\sqcup V_m))$, $L_2=(x_{m-1,1},\ldots,x_{m-1,\,t_{m-1}})$, $L^{0}=(0)$ and $L^{i}=(x_{m1}^{w_{m1}},x_{m2}^{w_{m2}},\ldots,x_{mi}^{w_{mi}})$ for $1\leq i\leq t_m-1$. Thus
$$\mbox{reg}\,(L^{0})=0,\ \mbox{reg}\,(L_2)=1,\  \mbox{and} \  \mbox{reg}\,(L^{i})=\sum\limits_{j=1}^{i}w_{mj}-(i-1) \ \ \mbox{for all}\ \ 1\leq i\leq t_m-1.$$
Note that the variables appearing in $L_{1}$, $L_2$ and $L^{i}$ are different from each other.
Therefore, it is  enough  to calculate $\mbox{reg}\,(L_{1})$ in order to compute $\mbox{reg}\,((J_{i}:x_{m,\,i+1}^{w_{m,i+1}}))$ by Lemma  \ref{lem4} (2).
 We distinguish into the following two cases:

(I) If $m=3$,  then $I(D\setminus (V_2\sqcup V_3))=(0)$. In this case, $L_1=(x_{11}^{w_{11}},x_{12}^{w_{12}},\ldots,x_{1,\,t_1}^{w_{1,t_1}})$. Thus
$$\mbox{reg}\,(L_1)=\sum\limits_{j=1}^{t_1}w_{1j}-t_1+1.$$
(II) If $m\geq4$.  In this case,
$L_1=(x_{11}^{w_{11}},x_{12}^{w_{12}},\ldots,x_{1,\,t_1}^{w_{1,t_1}}) +I(D\setminus (V_{m-1}\sqcup V_m))$.
Let $L_1^{\mathcal {P}}$ be the polarization of the ideal $L_1$, then it can be regarded as the  polarization of the edge ideal of a vertex-weighted oriented  graph $H$ with  whiskers, its vertex set $(\mathop\bigsqcup\limits_{j=1}^{m-2}V_j)\sqcup\{y_{11},\ldots,y_{1,\,t_1}\}$, edge set
$E(D\setminus (V_{m-1}\sqcup V_m))\cup \{x_{11}y_{11},\ldots, x_{1,\,t_1}y_{1,\,t_1}\}$ and the weight function  is $w': V(H)\rightarrow \mathbb{N}^{+}$ with $w'(x_{1j})=1$, $w'(y_{1j})=w_{1j}-1$ for $1\leq j\leq t_1$ and $w'(x)=w(x)$ for  any other vertex $x$. Thus $w'(x_{1j})+w'(y_{1j})=w_{1j}$  for $1\leq j\leq t_1$. By   Lemma \ref{lem6} and Theorem \ref{thm3}, we obtain
\[
\mbox{reg}\,(L_1)=\sum\limits_{x\in V(H)}w'(x)-|V(H\setminus W)|+1\
=\sum\limits_{\ell=1}^{m-2}(\sum\limits_{j=1}^{t_{\ell}}w_{\ell,\,j} )-\sum\limits_{j=1}^{m-2}t_{j}+1
\]
where $W=\{y_{11},\ldots,y_{1,\,t_1}\}$.

 Next we will prove  $\mbox{reg}\,((J_{i}:x_{m,\,i+1}^{w_{m,i+1}}))\leq \sum\limits_{x\in V(D)}w(x)-|V(D)|+1-w_{m,\,i+1}$, for $0\leq i\leq t_m-1$.

Since the variables that appear in $L_1$, $L_2$  and $L^{i}$ are different from each other  for any  $0\leq i\leq t_m-1$,  by  Lemma \ref{lem4} (2), we can get
\begin{eqnarray*}
\mbox{reg}\,((J_{0}:x_{m1}^{w_{m1}}))\!\!\!&=&\!\!\!\mbox{reg}\,(L_{1}+L_{2})=\mbox{reg}\,(L_1)+\mbox{reg}\,(L_{2})-1\\
&=&[\sum\limits_{\ell=1}^{m-2}(\sum\limits_{j=1}^{t_{\ell}}w_{\ell,\,j} )-\sum\limits_{j=1}^{m-2}t_{j}+1]+1-1\\
&=&\!\!\!\!\!\!\sum\limits_{x\in V(D)}\!\!w(x)-\!\!\sum\limits_{j=1}^{t_{m-1}}\!w_{m-1,\,j}-\!\!\sum\limits_{j=1}^{t_{m}}\!w_{mj}-|V(D)|+ t_{m-1}+t_m+1\\
&=&\!\!\!\!\!\!\sum\limits_{x\in V(D)}\!\!w(x)-|V(D)|+1+ t_{m-1}+t_m-\!\!\sum\limits_{j=1}^{t_{m-1}}\!w_{m-1,\,j}-\!\!\sum\limits_{j=1}^{t_{m}}\!w_{m,\,j} \\
&\leq&\!\!\!\!\!\!\sum\limits_{x\in V(D)}\!\!w(x)-|V(D)|+1+ t_{m-1}+t_{m}-2t_{m-1}-2(t_{m}-1)-w_{m1}\\
&\leq&\sum\limits_{x\in V(D)}w(x)-|V(D)|+1-w_{m1},
\end{eqnarray*}
and,  for $1\leq i\leq t_m-1$,
\begin{eqnarray*}
\mbox{reg}\,((J_{i}:x_{m,\,i+1}^{w_{m,i+1}}))\!\!\!&=&\!\!\!\mbox{reg}\,(L_{1}+L_{2}+L^{i}) =\mbox{reg}\,(L_{1})+\mbox{reg}\,(L_{2})+\mbox{reg}\,(L^{i})-2\\
\!\!\!&=&\!\!\!(\sum\limits_{\ell=1}^{m-2}(\sum\limits_{j=1}^{t_{\ell}}w_{\ell,\,j} )-\!\!\sum\limits_{j=1}^{m-2}t_{j}+1)+1+(\sum\limits_{j=1}^{i}w_{mj}-i+1)-2\\
\!\!\!&=&\!\!\!\!\!\!\sum\limits_{x\in V(D)}\!\!w(x)\!-\!|V(D)|\!+\!1\!+\!t_{m-1}+t_m-i-\!\!\!\sum\limits_{j=1}^{t_{m-1}}\!w_{m-1,\,j}- \!\!\!\sum\limits_{j=i+1}^{t_{m}}\!w_{m,\,j}\\
\!\!\!&\leq&\left\{\begin{array}{ll}
\sum\limits_{x\in V(D)}w(x)-|V(D)|+1-w_{m,\,t_m}\\
+t_{m-1}+t_m-i-2t_{m-1},\ \ &\text{if}\ \ i=t_m-1,\\
\sum\limits_{x\in V(D)}w(x)-|V(D)|+1-w_{m,\,i+1}\\
+i+2-t_{m-1}-t_m,\ \ &
\text{if}\ \ \  1\leq i\leq t_m-2.
\end{array}\right.\\
\!\!\!&\leq&\!\!\!\sum\limits_{x\in V(D)}w(x)-|V(D)|+1-w_{m,\,i+1}
\end{eqnarray*}
where the first inequality in the above formulas is  due to $w_{m-1,\,j}\geq 2$ for $1\leq j\leq t_{m-1}$, and $w_{mj}\geq 2$ for $i+2\leq j\leq t_m$.

\medskip
(2) Next we will prove $\mbox{reg}\,( J_{t_m})=\sum\limits_{x\in V(D)}w(x)-|V(D)|+1$, this implies that $\mbox{reg}\,((J_{i}:x_{m,\,i+1}^{w_{m,i+1}}))+w_{m,\,i+1}\leq\mbox{reg}\,( J_{t_m})\ \ \mbox{for all}\ \ 0\leq i \leq t_{m}-1$. Thus the assertion follows from Lemma \ref{lem2} (2) and by repeatedly using Lemma \ref{lem7} (1) on the short exact sequences $(\ddagger\ddagger\ddagger)$.

In fact,  we  write $J_{t_m}$ as
\[
J_{t_m}=I(D')+L,
\]
where $I(D')$ is the edge ideal of a vertex-weighted oriented subgraph $D'$ of  $D$, where $D'$ obtained from $D$ deleting the edges $\{x_{m-1,1}x_{m1},\ldots,x_{m-1,1}x_{m,\,t_m}, x_{m-1,2}x_{m1},\ldots,\\ x_{m-1,2}x_{m,\,t_m},\ldots,x_{m-1,\,t_{m-1}}x_{m1}, \ldots,x_{m-1,\,t_{m-1}}x_{m,\,t_m}\}$, and  $L=(x_{m1}^{w_{m1}},x_{m2}^{w_{m2}},\ldots,x_{m,\,t_m}^{w_{m,t_m}})$.
Then the polarization $J_{t_m}^{\mathcal {P}}$ of  ideal $J_{t_m}$  can be regarded as the  polarization of the edge ideal of a vertex-weighted oriented graph $D''$ with whiskers, its vertex set $(\mathop\bigsqcup\limits_{j=1}^{m}V_j)\sqcup\{y_{m1},\ldots,y_{m,\,t_m}\}$, edge set
$E(D')\cup \{x_{m1}y_{m1},\ldots, x_{m,\,t_m}y_{m,\,t_m}\}$ and the  weight function  is $w'': V(D'')\rightarrow \mathbb{N}^{+}$ with $w''(x_{mj})=1$, $w''(y_{mj})=w_{mj}-1$ for $1\leq j\leq t_m$ and $w''(x)=w(x)$ for  any other vertex $x$. Thus $w''(x_{mj})+w''(y_{mj})=w_{mj}$  for $1\leq j\leq t_m$. By   Lemma \ref{lem6} (2) and Theorem \ref{thm3}, we obtain
\begin{eqnarray*}
\mbox{reg}\,( J_{t_m})&=&\sum\limits_{x\in V(D'')}w''(x)-|V(D''\setminus V'')|+1=\sum\limits_{\ell=1}^{m}(\sum\limits_{j=1}^{t_{\ell}}w_{\ell,\,j})-\sum\limits_{j=1}^{m}t_{j}+1\\
&=&\sum\limits_{x\in V(D)}w(x)-|V(D)|+1
\end{eqnarray*}
where $V''=\{y_{m1},\ldots,y_{m,\,t_m}\}$.
This proof is complete.
\end{proof}

\medskip

\begin{Theorem}\label{thm5}
Let $D=(V(D),E(D),w)$ be a  vertex-weighted oriented  graph as Theorem \ref{thm4}. Then
 $$\mbox{pd}\,(I(D))=|V(D)|-1.$$
\end{Theorem}

\begin{proof}
Let $V_{i}=\{x_{i1},\ldots,x_{i,\,t_{i}}\}$ for $1\leq i\leq m$. Then
\begin{eqnarray*}
I(D)\!\!\!&=&\!\!\!(x_{11}x_{21}^{w_{21}}\!,x_{11}x_{22}^{w_{22}}\!,\ldots,x_{11}x_{2,\,t_2}^{w_{2,t_2}}\!, x_{12}x_{21}^{w_{21}}\!,\ldots,x_{12}x_{2,\,t_2}^{w_{2,t_2}}\!,\ldots,x_{1,\,t_1}x_{21}^{w_{21}}\!,\ldots, \\
& & x_{1,t_1}x_{2,\,t_2}^{w_{2,t_2}}\!,x_{21}x_{31}^{w_{31}}\!,\ldots, x_{21}x_{3,\,t_3}^{w_{3,t_3}}\!,x_{22}x_{31}^{w_{31}}\!,\ldots, x_{22}x_{3,\,t_3}^{w_{3,t_3}}\!,\ldots,  x_{2,\,t_2}x_{31}^{w_{31}}\!, \ldots,\\
& &x_{2,\,t_2}x_{3,\,t_3}^{w_{3,t_3}}\!,\ldots,x_{m-1,1}x_{m1}^{w_{m1}},\ldots,x_{m-1,1}x_{m,\,t_m}^{w_{m,t_m}}\!, x_{m-1,2}x_{m1}^{w_{m1}}\!,\ldots, x_{m-1,2}x_{m,\,t_m}^{w_{m,t_m}}\!,\\
& &\ldots,x_{m-1,\,t_{m-1}}x_{m1}^{w_{m1}}\!, \ldots,x_{m-1,\,t_{m-1}}x_{m,\,t_m}^{w_{m,t_m}}\!,x_{m1}x_{11}^{w_{11}}\!, \ldots, x_{m1}x_{1,\,t_1}^{w_{1,t_1}}\!,x_{m2}x_{11}^{w_{11}}\!, \\
& & \ldots,x_{m2}x_{1,\,t_1}^{w_{1,t_1}}\!,\ldots, x_{m,\,t_{m}}x_{11}^{w_{11}}\!,
\ldots,x_{m,\,t_{m}}x_{1,\,t_1}^{w_{1,t_1}}).
\end{eqnarray*}
Consider the following short exact sequences
\begin{gather*}
\begin{matrix}
 0 & \longrightarrow & \frac{S}{(I(D)\,:\,x_{m1})}(-1) & \stackrel{ \cdot x_{m1}} \longrightarrow & \frac{S}{I(D)} &\longrightarrow & \frac{S}{J_1} & \longrightarrow & 0 \\
 \\
 0 & \longrightarrow & \frac{S}{(J_1\,:\,x_{m2})}(-1) & \stackrel{ \cdot x_{m2}} \longrightarrow  & \frac{S}{J_1} & \longrightarrow & \frac{S}{J_2} & \longrightarrow & 0 &\hspace{1.0cm} (\ddagger\ddagger\ddagger\ddagger) \\
 \\
     &  &\vdots&  &\vdots&  &\vdots&  &\\
     \\
 0&  \longrightarrow & \frac{S}{(J_{t_{m}-1}\,:\,x_{m,\,t_{m}})}(-1)& \stackrel{ \cdot x_{{m},\,t_{m}}} \longrightarrow & \frac{S}{J_{t_{m}-1}}& \longrightarrow & \frac{S}{J_{t_{m}}}& \longrightarrow & 0\\
 \end{matrix}\quad
\end{gather*}
where $J_{i}=I(D)+(x_{m1},x_{m2},\ldots,x_{mi})$ \ \  for $1\leq i\leq t_m$.  We  prove this argument into the following two steps.

(1) We first prove  $\mbox{pd}\,(J_{i}:x_{m,\,i+1})=|V(D)|-1$ for  all $0\leq i\leq t_m-1$, where $J_0=I(D)$.
We  write $(J_{i}:x_{m,\,i+1})$ as follows:
\[
(J_{i}:x_{m,\,i+1})=L_{1}^{i}+L_{2}^{i}
\]
where $L_{1}^{0}=(0)$, $L_{1}^{i}=(x_{m1},x_{m2},\ldots,x_{mi})$ for $1\leq i\leq t_m-1$, $L_{2}^{i}=(x_{11}^{w_{11}},x_{12}^{w_{12}},\ldots,\\ x_{1,\,t_1}^{w_{1,t_1}}) +I(D_{i})$, and  $D_{i}$ is an induced subgraph $D\setminus\{x_{m1},x_{m2},\ldots,x_{mi}\}$ of $D$ on the set  $V\setminus\{x_{m1},x_{m2},\ldots,x_{mi}\}$.

Since the variables appearing in $L_{1}^{i}$ and $L_2^{i}$  are different and $\mbox{pd}\,(L_1^{i})=i-1$,
we only need  to calculate $\mbox{pd}\,(L_2^{i})$ in order to compute $\mbox{pd}\,((J_{i}:x_{m,\,i+1}))$ by Lemma  \ref{lem4}.

For $0\leq i\leq t_m-1$, the polarization $(L_2^{i})^{\mathcal {P}}$  of the ideal $L_2^{i}$  can be regarded as the  polarization of the edge ideal of a vertex-weighted  oriented graph $H_i$ with whiskers, its vertex set $V(D_i)\sqcup\{y_{11},\ldots,y_{1,\,t_1}\}$,  edge set
$E(D_i)\cup \{x_{11}y_{11},\ldots, x_{1,\,t_1}y_{1,\,t_1}\}$ and the weight function  is $w_i: V(H_i)\rightarrow \mathbb{N}^{+}$ with $w_{i}(x_{1j})=1$, $w_{i}(y_{1j})=w_{1j}-1$, $w_i(x_{m,\,i+1})=w_{m,\,i+1}-1$ for $1\leq j\leq t_1$ and $w_i(x)=w(x)$ for  any other vertex $x$. In fact, $H_i$ is an oriented graph as  Theorem \ref{thm3}. Then  by Lemma \ref{lem6} (3) and Theorem \ref{thm3}, we have
\begin{eqnarray*}
\mbox{pd}\,(J_{0}:x_{m,1})&=&|V(D)|-1,\\
\mbox{pd}\,(J_{i}:x_{m,\,i+1})&=&\mbox{pd}\,(L_{1}^{i}+L_{2}^{i})=\mbox{pd}\,(L_{1}^{i})+\mbox{pd}\,(L_{2}^{i})+1\\
&=&(i-1)+(|V(D)|-i-1)+1=|V(D)|-1.
\end{eqnarray*}

(2) Next we will compute  $\mbox{pd}\,(J_{t_m})\leq |V(D)|-2$. Thus we have  $\mbox{pd}\,(J_{t_{m}})< \mbox{pd}\,(J_{i}:x_{m,\,i+1})\ \ \mbox{for all}\ \ 0\leq i \leq t_{m}-1$.  Therefore, the  assertion follows from Lemma \ref{lem2} (1) and by repeatedly using Lemma \ref{lem7} (2) on the short exact sequences $(\ddagger\ddagger\ddagger\ddagger)$.

In fact, we  notice that $$J_{t_m}=L_1+L_2$$
where $L_{1}=(x_{m1},x_{m2},\ldots,x_{m,t_m})$ and $L_2=I(D\setminus V_m)$.
Notice that $L_2$ is the edge ideal of the induced subgraph $D\setminus V_m$ of $D$, it is
a vertex-weighted  $(m-1)$-partite  graph  with the vertex set $\mathop\bigsqcup\limits_{i=1}^{m-1}V_{i}$. Using  Theorem \ref{thm2} and Lemma \ref{lem4} (1), we obtain
\begin{eqnarray*}
\mbox{pd}\,(J_{t_m})\!\!&=&\!\!\mbox{pd}\,(L_{1})+\mbox{pd}\,(L_{2})+1=(t_m-1)+(|V(D\setminus V_m)|-2)+1\\
&=&t_m+(|V(D)|-t_m)-2=|V(D)|-2.
\end{eqnarray*}
The proof is complete.
\end{proof}

\medskip
The following theorem generalizes Theorem 5.1 of \cite{Z3}.
\begin{Corollary}\label{cor5}
Let $D=(V(D),E(D),w)$ be a weighted oriented cycle such that  $w(x)\geq 2$  for any  vertex $x$. Then
\begin{itemize}
\item[(1)] $\mbox{reg}\,(I(D))=\sum\limits_{x\in V(D)}w(x)-|E(D)|+1$,
\item[(2)] $\mbox{pd}\,(I(D))=|E(D)|-1$.
\end{itemize}
\end{Corollary}
\begin{proof} Let $V(D)=\{x_1,\ldots,x_n\}$. Then $D$ is an oriented $n$-partite  graph as Theorem \ref{thm4} with vertex set $V=\mathop\bigsqcup\limits_{i=1}^{n}V_i$,  where $V_i=\{x_i\}$,
 and   edge set $E=\bigcup\limits_{i=1}^{m}E(D_{i})$, where $E(D_i)=\{x_ix_{i+1}\}$. Thus $|E(D)|=|V(D)|=n$ and the assertion follows from two theorems  above.
\end{proof}

\medskip
The following  corollaries are  immediate consequences of   two theorems above.

\begin{Corollary}\label{cor6}
Let $D=(V(D),E(D),w)$ be a weighted oriented complete tripartite graph such that  $w(x)\geq 2$  for any  vertex $x$. Then
\begin{itemize}
\item[(1)] $\mbox{reg}\,(I(D))=\sum\limits_{x\in V(D)}w(x)-|V(D)|+1$,
\item[(2)] $\mbox{pd}\,(I(D))=|V(D)|-1$.
\end{itemize}
\end{Corollary}

\begin{Corollary}\label{cor7}
Let $D=(V(D),E(D),w)$ be a vertex-weighted  oriented graph as Theorem \ref{thm4}. Then
$$\mbox{depth}\,(I(D))=1$$
\end{Corollary}
\begin{proof}
It follows from Auslander-Buchsbaum formula and the above theorem.
\end{proof}

\medskip
The following example shows that the projective dimension and the
regularity  of the edge ideals of vertex-weighted oriented  graphs as Theorem \ref{thm4} are related to direction selection.
\begin{Example}  \label{example8}
Let $I(D)=(x_1x_3^{3},x_2x_3^{3},x_3x_4^{3},x_1x_4^{3},x_2x_4^{3})$ be the edge ideal of weighted oriented $3$-partite graph $D=(V(D),E(D),w)$ with $w_1=w_2=1$ and  $w_3=w_4=3$,  where  $V=\mathop\bigsqcup\limits_{j=1}^{3}V_j$ with
$V_{1}=\{x_1, x_2\}$, $V_{2}=\{x_3\}$ and $V_{3}=\{x_4\}$. By using CoCoA, we obtain $\mbox{reg}\,(I(D))=6$ and $\mbox{pd}\,(I(D))=2$. But we have
 $\mbox{reg}\,(I(D))=\sum\limits_{i=1}^{4}w_i-|V(D)|+1=5$ by Theorem \ref{thm4}
and $\mbox{pd}\,(I(D))=|V(D)|-1=3$
 by Theorem \ref{thm5}.
 \end{Example}

\medskip
The following example shows that the assumption  that  $w(x)\geq 2$  for any  $x\in V(D)$ in Theorem \ref{thm4} and Theorem \ref{thm5} cannot be dropped.
\begin{Example}  \label{example2}
Let $I(D)=(x_1x_2^{2},x_1x_3^{2},x_1x_4^{2},x_2x_5^{3},x_3x_5^{3},x_4x_5^{3},x_5x_1)$ be the edge ideal of vertex-weighted oriented tripartite graph $D=(V,E,w)$ with $w_1=1$, $w_2=w_3=w_4=2$ and  $w_5=3$, where $V=\mathop\bigsqcup\limits_{j=1}^{3}V_j$ with $V_{1}=\{x_1\}$, $V_{2}=\{x_2,x_3,x_4\}$ and  $V_{3}=\{x_5\}$. By using CoCoA, we obtain  $\mbox{reg}\,(I(D))=5$ and $\mbox{pd}\,(I(D))=3$. But we have
$\mbox{reg}\,(I(D))=\sum\limits_{i=1}^{5}w_i-|V(D)|+1=6$
 by Theorem \ref{thm4}
and $\mbox{pd}\,(I(D))=|V(D)|-1=4$ by Theorem \ref{thm5}.
\end{Example}

\medskip

\medskip
\hspace{-6mm} {\bf Acknowledgments}

 \vspace{3mm}
\hspace{-6mm}  This research is supported by the National Natural Science Foundation of China (No.11271275) and  by foundation of the Priority Academic Program Development of Jiangsu Higher Education Institutions.

\medskip

\end{document}